\documentclass[oneside]{amsart}

\numberwithin{equation}{section}
\newcommand{\be}{\begin{equation}}
\newcommand{\ee}{\end{equation}}

\newcommand{\<}{\langle}
\renewcommand{\>}{\rangle}

\newcommand{\R}{\mathbb R}
\newcommand{\HH}{\mathbb H}

\newcommand{\LL}{\mathcal L}
\newcommand{\U}{\mathcal U}
\newcommand{\B}{\mathcal B}

\newcommand{\ep}{\varepsilon}

\renewcommand{\phi}{\varphi}
\newcommand{\pd}{\partial}
\newcommand{\ov}{\overline}

\newcommand{\II}{\mathbf{I\!I}}

\DeclareMathOperator{\grad}{grad}
\DeclareMathOperator{\trace}{trace}
\DeclareMathOperator{\vol}{vol}
\def\div{\operatorname{div}}

\def\dist{\operatorname{dist}}

\newtheorem{theorem}{Theorem}[section]
\newtheorem{proposition}[theorem]{Proposition}
\newtheorem{corollary}[theorem]{Corollary}
\newtheorem{lemma}[theorem]{Lemma}

\newtheorem{conj}[theorem]{Conjecture}

\theoremstyle{definition}

\theoremstyle{definition}
\newtheorem{definition}[theorem]{Definition}
\newtheorem{df}[theorem]{Definition}
\newtheorem{notation}[theorem]{Notation}

\theoremstyle{remark}
\newtheorem*{remark}{Remark}

\begin{document}

\title{Area minimizers and boundary rigidity of almost hyperbolic metrics}

\author{Dmitri Burago}                                                          
\address{Dmitri Burago: Pennsylvania State University,                          
Department of Mathematics, University Park, PA 16802, USA}                      
\email{burago@math.psu.edu}                                                     
                                                                                
\author{Sergei Ivanov}
\address{Sergei Ivanov:
St.Petersburg Department of Steklov Mathematical Institute,
Fontanka 27, St.Petersburg 191023, Russia}
\email{svivanov@pdmi.ras.ru}

\thanks{The first author was partially supported                                
by NSF grants DMS-0604113 and DMS-0412166.                                      
The second author was partially supported by the                                
Dynasty foundation and RFBR grants 08-01-00079-a and 09-01-12130-ofi-m.}

\begin{abstract}
This paper is a continuation of our paper about boundary rigidity and filling
minimality of metrics close to flat ones.  We show that compact regions close 
to a hyperbolic one are boundary distance rigid and strict minimal fillings. We also
provide a more invariant view on the approach used in the above 
mentioned paper. 
\end{abstract}


\maketitle

\section{Introduction and preliminaries}

\subsection{Minimal surfaces vs. area minimizes: a preliminary discussion.}

Before we proceed to the main results of this paper, we begin with a very
general consideration. We want to formulate
some open problems followed by a brief discussion. We begin with the
following problem, which sounds extremely natural, but we do not know
any reference for it.

Let $M^N$ be a complete Riemannian manifold and $S$ a compact $n$-dimensional
surface in $M$ with $\pd S\ne\emptyset$. Assume that $S$ is a convex set
 in the strongest possible sense, namely for every two points
$p,q \in S$, there is a unique shortest path between $p$ and $q$ in $M$
and this shortest path lies in $S$. Is it true that $S$ is an area minimizer,
namely that for any other surface $S_1$ such that 
$\partial S_1=\partial S$, the $n$-dimensional area $\vol_nS_1$ of $S_1$
is greater than that of $S$?

At first glance, this seems to be a naive question that should be easy.
However, apparently it is not and a solution would imply solutions of
some notoriously difficult problems. 

We want to emphasize the difference between minimal
surfaces (in the variational sense), local area minimizers
(that is, minimizing the area among all nearby surfaces)
and global area minimizers. In the above question,
$S$ is totally geodesic and hence is a minimal surface.
Moreover one can show that (any proper subregion of) $S$ is a local minimizer
(we do not give a proof here since we do not use this fact).
The problem begins when we are looking
for global minimality.

We know only two general methods of proving global minimality:
constructing a projection or a calibrating form. 

A projection is a map $P\colon M \to S$ which is the
identity on $S$ and does not increase $n$-dimensional areas.
This condition is equivalent to saying
that the $n$-dimensional Jacobian of $P$
is no greater than 1 everywhere on $M$.
(If in addition this Jacobian is strictly less than 1
outside $S$, the projection guarantees that $S$ is a
unique area minimizer among the surfaces with the
same boundary.)
It is very unlikely that such a projection
exists for all minimizers even if there are no topological obstructions.

A more general method is constructing a calibration form,
that is a closed $n$-form $\omega$ on $M$
such that the restriction of $\omega$ to $S$ is the $n$-dimensional
volume form of $S$ and the norm of $\omega$ is less than 1
outside $S$. Then by Stokes' Formula we immediately solve
the problem (for orientable surfaces). 

However there is a difficulty with both methods. Imagine that
there  exists a surface $S_2$ such that its boundary
$\pd S_2=10\cdot\pd S$ (that is, $\pd S_2$
covers $\pd S$ 10 times)
and $\vol_nS_2 < 10\vol_nS$. In this case we say that $S$ is not
stably minimizing. 
Such a surface can be an area minimizer but none of the above two methods
can prove this.
Such phenomena take place in
a situation rather close to the one that will be discussed in this paper.
Namely, regions of affine subspaces in normed spaces can be
global minimizers but not stably minimizing surfaces (with respect 
to the Holmes-Thompson surface area \cite{HT}), see \cite {BI02} and \cite{BI04}.

\subsection{Boundary rigidity and minimal fillings}
Now let us present our general set-up and explain why it is directly 
related to the above discussion. This paper is a continuation of
\cite{BI} and we borrow some parts of introduction and formulations from 
there. We hope that we also can give a better and more invariant insight
into what is done in \cite{BI}, especially due to Proposition  \ref{p:mean-curvature-zero}.

Let $M=(M^n,g)$ be a compact Riemannian manifold with boundary $\pd M$.
Its {\it boundary distance function}, denoted by $bd_M$, is the restriction of the
Riemannian distance $d_M$ to $\pd M\times\pd M$. The term
``boundary rigidity'' means that the metric is uniquely determined
by its boundary distance function. More precisely,

\begin{df} $M$ is {\it boundary rigid} if every
compact Riemannian manifold $M'$ with the same boundary and
the same boundary distance function is isometric to $M$ via a
boundary preserving isometry.
\end{df}

It is easy to construct metrics that are not boundary rigid. For
example, consider a metric on a disc with a ``big bump'' around a
point $p$, such that the distance from $p$ to the boundary is
greater than the diameter of the boundary. Since no minimal
geodesic between boundary points passes through $p$, a
perturbation of the metric near $p$ does not change the boundary
distance function.

Thus one has to impose restrictions on the metric in order to make
the boundary rigidity problem sensible. One natural restriction is
the following: a Riemannian manifold $M$ is called {\it
simple} if the boundary $\pd M$ is strictly convex, every two
points $x,y\in M$ are connected by a unique geodesic, and
geodesics have no conjugate points (cf.~\cite{Michel}). A more
general condition called SGM (``strong geodesic minimizing'') was
introduced in \cite{Croke91} in order to allow non-convex
boundaries. Note that if $M$ is simple, then it is a
topological disc.

The simplicity of $M$ can be seen from the
boundary distance function. Indeed, the convexity of $\pd M$ is equivalent to
a (local) inequality between boundary distances and intrinsic
distances of $\pd M$. And if one already knows that the boundary is convex,
then simplicity is equivalent to
$C^2$-smoothness of the boundary distance function away from the diagonal.
(This is an easy folklore fact which we do not use and leave
to the reader.)
Thus if two Riemannian
manifolds have the same boundary and the same boundary distance
functions, then either both are simple or both are not.

\begin{conj} [Michel \cite{Michel}]
All simple manifolds are boundary rigid.
\end{conj}

Pestov and Uhlmann \cite{PU} proved this conjecture in
dimension~2. In higher dimensions, few examples of boundary rigid
metrics are known. They are: regions in $\R^n$ \cite{Gromov},
in the open hemisphere \cite{Michel}, in symmetric spaces of
negative curvature (follows from the main result of \cite{BCG}),
and in product spaces of the form $N\times\R$ where $N$
is a simple $(n-1)$-dimensional Riemannian manifold \cite{CK98}.
We refer the reader to \cite{Croke04} and \cite{PU}
 for a survey of boundary rigidity,
other inverse problems, and their applications.

One of the main results of \cite{BI} asserts that
if $M$ is sufficiently
close to a region in the Euclidean space, then $M$ is
boundary rigid. In this paper we extend this to metrics
close to the hyperbolic one.
Namely we prove the following theorem:

\begin{theorem}\label{t:rigid}
If a Riemannian metric in a region is sufficiently close to a hyperbolic metric, then 
it is boundary rigid.  More precisely, let $D\subset\HH^n$ be a
compact region with a smooth boundary.
The there is a $C^r$-neighborhood (for a suitable $r$)
of the standard hyperbolic metric on $D$
such that for every metric $g$ from this neighborhood,
the space $M=(D,g)$ is boundary rigid.
\end{theorem}

Here and below by ``region'' we mean a connected set
bounded by a smooth submanifold.

\begin{remark}
We do not track the number of derivatives required
for our arguments to work.
Therefore the formal meaning of the word ``smooth''
in this paper is $C^\infty$.
An interested reader can
verify that one can take $r=3$ for ``a suitable $r$''
in  Theorem~\ref{t:rigid}. This is worse than in \cite{BI}
where the metric has to be only $C^2$-close to the
Euclidean one.

The proof of Theorem \ref{t:rigid} can be also made to work
for metrics close to the Euclidean one
(in fact, the proof in this case is much easier),
an outline of the argument can be found in \cite{I10}.
While this approach proves a slightly weaker result
than in \cite{BI} (namely $C^2$ is replaced by $C^3$),
it provides a better insight into and allows some 
simplifications of the methods of \cite{BI}.
\end{remark}

We treat boundary rigidity  as the
equality case of the minimal filling problem discussed
in \cite {BI}, \cite{BI02}, and \cite{Ivanov}.

\begin{df}
\label{mf}
$M$ is a {\it minimal filling} if,
for every compact Riemannian manifold $M'$ with $\pd M'=\pd M$,
the inequality
$$
  d_{M'}(x,y) \ge d_M(x,y) \quad \text{for all $x,y\in\pd M$}
$$
implies
$$
  \vol(M') \ge \vol(M) .
$$
We say that $M$ is a \textit{strict minimal filling} if in addition the equality
$$
  \vol(M') = \vol(M)
$$
implies that $M$ and $M'$ are isometric via an
isometry that is identical on the boundary.
\end{df}

\begin{conj}
\label{filling}
Every simple manifold is a strict minimal filling.
\end{conj}

If $M$ is simple, then its volume is
uniquely determined by its boundary distance function,
namely there is an integral formula
expressing $\vol(M)$ in terms of $bd_M$ and its first order derivatives
(the Santal\'o formula \cite{Santalo}, see also \cite{Gromov}). It
is not clear though whether the formula is monotone in $bd_M$.

However the mere existence of this formula implies
that every simple strict minimal filling is boundary rigid.
Indeed, let $M$ be simple and a strict minimal filling,
and suppose that $M'$ has the same boundary and the same boundary
distance function as~$M$. Since $M'$ has the same boundary distance
function as a simple manifold, it is also simple,
hence $\vol(M')=\vol(M)$ by the Santal\'o formula.
Now the equality case in Definition \ref{mf} implies
that $M'$ and $M$ are isometric.

Our approach to boundary rigidity is to prove
a suitable partial case of Conjecture \ref{filling}.
In \cite{BI} we were able to carry out this plan for
metrics close to a Euclidean one. Now we extend this method to metrics
close to a hyperbolic one.
Namely, we prove the following theorem:

\begin{theorem}\label{minimality-thm}
If a metric $g$ on a region $D \subset \HH^n$ is sufficiently close to the 
hyperbolic one, then $(D,g)$ is a strict minimal filling
(and hence is boundary rigid).
\end{theorem}

Let us note that one of the basic difficulties in the boundary rigidity problem 
is that boundary rigidity fails for Finsler manifolds under even very strong assumptions 
(see e.g.\ an example in the introduction of \cite{BI}) hence one has to substantially use the 
fact that we are dealing with Riemannian manifolds. Nevertheless, we have no idea whether
filling volume minimality holds in the Finsler world (beyond dimension two). Our proof 
is motivated by Finsler geometry and its structure is basically Finsler; however, when
we introduce an auxiliary structure of an $L^2$-manifold on $L^\infty$ and start 
computation of Jacobians relating volumes, the proof becomes essentially Riemannian. 

\subsection{Plan of the proof}
We deduce filling minimality of a metric from area minimality of its
image in a suitable space. More precisely, given $M=(D,g)$ as in
Theorem \ref{minimality-thm}, we do the following:

(i) Construct a distance-preserving map $\Phi:M\to\LL=L^\infty(S)$ where
$S$ is a suitable measure space. Following the tradition of Geometric Measure Theory,
we refer to Lipschitz maps from manifolds of any dimension to a normed space as 
{\em surfaces}.  

(ii) Define an $n$-dimensional surface area functional
(referred to as the $n$-volume) for Lipschitz surfaces
in $\LL$ so that the following holds. First, every 1-Lipschitz map
from a Riemannian manifold to $\LL$ does not increase $n$-volumes.
Second, the above map $\Phi$  preserves $n$-volumes.

(iii) Prove that $\Phi$ (regarded as a surface in $\LL$)
is an area minimizer, that is it has
the least $n$-volume among
all Lipschitz surfaces with the same boundary in $\LL$. 
Moreover, $\Phi$ is a unique area-minimizer: every
Lipschitz surface in $\LL$ having the same boundary
and the same $n$-volume as $\Phi$ is contained in the
image of $\Phi$ (up to a set of zero area).

If there exist $\LL$, $\Phi$ and a surface area functional satisfying (i)--(iii),
then $M$ is a strict minimal filling. Indeed, consider a Riemannian
$n$-manifold $M'$ such that $\pd M'=\pd M$ and
$d_{M'}\ge d_M$ on $\pd M\times\pd M$. The latter
implies that the map $\Phi|_{\partial M'}\to\LL$ is 1-Lipschitz
(with respect to the distance $d_{M'}$ on $\partial M'$).
It is easy to see (cf.\ e.g.\ \cite[Proposition 1.6]{Iv09}) that $\LL=L^\infty(S)$ enjoys
the following Lipschitz extension property: for any metric space $X$,
any subset $Y\subset X$ and any 1-Lipschitz map $f:Y\to\LL$
there exists a 1-Lipschitz map $\tilde f:X\to\LL$ extending $f$.
Substituting $X=M'$, $Y=\pd M'$ and $f=\Phi|_{\pd M'}$
yields a 1-Lipschitz map $\Phi':M'\to\LL$ such that
$\Phi'|_{\pd M'}=\Phi|_{\pd M}$. Then by (ii) and (iii) we have
$$
 \vol(M') \ge \vol(\Phi') \ge \vol(\Phi)=\vol(M) ,
$$
hence $M$ is a minimal filling. To prove strict filling minimality,
observe that the equality in the above inequality implies that
$\Phi'$ preserves $n$-volumes and the image of $\Phi'$
is contained in the image of $\Phi$. Thus there is a
1-Lipschitz volume-preserving map $\Phi^{-1}\circ\Phi'$
from $M'$ to $M$. Such a map must be an isometry,
thus $M$ is a strict minimal filling.

Actually our proof contains some extra technical details
(in particular, we do not assume convexity of the boundary
and work  in some large ball in $\LL$ rather
than in the entire space). The proof with all these details
is contained in Section \ref{sec-proof-thm}.

For the purpose of proving boundary rigidity only,
this strategy {\em should} work as follows. Imagine that for
every simple manifold $M$ we construct (in a canonical way)
an isometric embedding $\Phi_{M}:M \to\LL=L^\infty(S)$
 with the following properties:

1. The restriction of $\Phi_{M}$ to $\pd M$ depends only on the
boundary distance  function of~$M$.

2. A notion of $n$-volume in $\LL$ is defined so that
these isometric embeddings $\Phi_M$ preserve $n$-volumes.

3. $\Phi_{M}$ is a strict area minimizer with respect to 
this $n$-volume in $\LL$ (at least for the particular manifold $M$
for which we are trying to prove rigidity).

Given such a construction, the boundary rigidity of $M$
follows immediately from the fact that the boundary distance
function uniquely determines the volume. Our proof
is more complicated than this, since our definition of the
$n$-volume in $\LL$ depends on $M$
(still we have that $\Phi_M$ preserves $n$-volumes and any competitor
map $\Phi_{M'}$ does not increase $n$-volumes).

There are many natural constructions satisfying the first
of the above requirements.
For instance, one could 
embed $M$ into $L^\infty(\pd M)$ by distance functions:
for $x\in M$, let $\Phi(x)=d_M(x,\cdot)|_{\pd M}$.
We use a slightly different embedding of the
same type where $\LL$ will be the $L^\infty$ of the ideal boundary
of $\HH^n$ rather than that of the boundary of $M$. 

The surface area functional in $\LL$ is defined by a suitable
Riemannian structure
(that is, a family of $L^2$-compatible scalar products) on $\LL$.
The strict minimality of the surface $\Phi_M$ is proved
by constructing a projection from $\LL$ to this surface
that strictly decreases $n$-volumes outside the surface.
Note that a proof following this strategy never mentions the competitor
manifold $M'$ and works exclusively with our simple manifold $M$.
 
A very important observation  (Proposition \ref{p:mean-curvature-zero}) is that for
a very natural variety of choices of $n$-volume forms on $\LL=L^\infty(S)$
(induced by certain choices of Riemannian structures on $\LL$),
$\Phi_M$ is a minimal surface (in the variational sense).
This
observation explains why there is a reasonable hope that
this strategy could work. The difficult part is to
prove the global area minimality of $\Phi_M$,
and this is where we use the assumption that
the metric of $M$ is close to a hyperbolic one.

In fact, we first prove the minimality for the hyperbolic case.
We do this by giving an explicit formula for a distance-preserving
embedding $\Phi\colon\HH^n\to\LL$ and an $n$-area contracting
retraction $P\colon\LL\to\Phi(\HH^n)$ 
(we call such a map ``projection'', see the beginning of the introduction).
This formula is a guesswork.
Moreover we show that the global area minimality of regions in $\Phi(\HH^n)$
is quadratically non-degenerate in a sense
(by using a compression trick, see Section~\ref{sec-compression}).
Then the case of an almost hyperbolic metric follows from
an infinite-dimensional generalization of the following standard fact:
In finite dimension, non-degenerate global area minimality persists under small perturbations
within the class of minimal surfaces. Generalizing this fact becomes somewhat
technically cumbersome since we work with both $L^2$ and $L^\infty$ and
there are different topologies involved. 

\medskip
\noindent{\it Note on notations}.
For technical reasons, we work not with a compact disc
with boundary but rather with a complete Riemannian manifold
containing this disc.
In order not to make the reader confused we mention here
that in the proofs below $M$ denotes 
that complete manifold rather than the original disc.

\begin{remark}
To define $n$-volumes in $\LL$ we
use a somewhat artificial construction
when a class of volumes is defined after we have an embedding of $M$.
There are several nice notions of $n$-volumes for surfaces in $L^\infty$, which,
unlike in our construction, do not depend on any base embedding or such.
Unfortunately, we were not able to use any of those nice
volumes directly. 
\end{remark}

\noindent
\textit{Acknowledgement}.
We are grateful to three anonymous referees for thoroughly reading our paper 
and for numerous very useful remarks, corrections, and suggestions. We are also 
grateful to Chris Croke, Bruce Klieiner,
Bill Minicozzi, and Gunther Uhlmann for interesting discussions.

\section{Proof of the theorems}
\label{sec-proof-thm}

The purpose of this section is to deduce the statements of Theorems
from Proposition \ref{main prop}. This proposition asserts  that there
exist two maps with certain easy-to formulate properties. The rest of 
the paper is a rather technical construction of the maps along with verification
of the properties.

Let $g_0$ denote the standard metric on $\HH^n$.
Let $D\subset\HH^n$ be a region with a smooth boundary
and $g$ a Riemannian metric on $D$ (assumed to be
close to ${g_0}_{|_D}$ in a suitable $C^r$ topology).

Fix a point $o\in\HH^n$, this point will be referred
to as the \textit{origin}. By $B_o(R)$ we denote
the ball of radius $R$ in $\HH^n$ centered at~$o$.
Fix an $R>0$ such that $D\subset B(R/5)$.
The metric $g$ can be extended from $D$ to $\HH^n$
so that the extension is smooth and coincide with $g_0$
outside $B(R/2)$. Moreover the extension can be
constructed in such a way that it converges to $g_0$
as  $g$ converges to $g_0|_D$.

We denote the extension by the same letter $g$
and denote $M=(\HH^n,g)$. 
Our goal is to prove that for any region
$D\subset B(R/2)$ the space
$(D,g)\subset M$ is a minimal filling and boundary rigid.

Let $S=S^{n-1}$ and  $\LL=L^\infty(S)$. 
For $r>0$, let $\B(r)$ denote
the ball of radius $r$ in $\LL$ centered at the origin.
Let $\B=\B(R)$ where $R$ is the radius fixed above.

The technical results
established in the rest of the paper can
be summarized as the following proposition.

\begin{proposition}
\label{main prop}
If $g$ is sufficiently close to $g_0$,
then there exists a distance-preserving map
$\Phi:M\to\LL$ such that $\Phi(o)=0$ and a Lipschitz map
$$
P_\sigma\colon\B\cup\Phi(M)\to M
$$
such that the following conditions hold:

1. $P_\sigma\circ\Phi=id_M$.

2. For every Riemannian $n$-manifold $N$ and
every 1-Lipschitz map $f:N\to\B$, the composition
$P_\sigma\circ f:N\to M$ does not increase $n$-dimensional
(Riemannian) volumes.

3. If, for $N$ and $f$ as above, the composition $P_\sigma\circ f:N\to M$
preserves volumes of all measurable sets,
then $f(N)\subset\Phi(M)$.
 \end{proposition}

Now we deduce the theorems from this proposition.

\begin{proof}[Proof of Theorem \ref{minimality-thm}]
Let $g$ be sufficiently close to $g_0$ so that
the maps $\Phi$ and $P_\sigma$ from
Proposition \ref{main prop} exist.
Let $g'$ be a metric on $D$ such that
$$
 d_{(D,g')}(x,y) \ge d_{(D,g)}(x,y)  
$$
for all $x,y\in\pd D$.
Denote $M'=(D,g')$.
Since $(D,g)\subset M$, we have $d_{(D,g)}(x,y)\ge d_M(x,y)$
for all $x,y\in D$, hence
$$
 d_{M'}(x,y) \ge d_M(x,y)  
$$
for all $x,y\in\pd D$.
Since $\Phi$ is distance-preserving with respect to $d_M$,
it follows that the restriction $\Phi|_{\pd D}$ is 1-Lipschitz
with respect to the metric $d_{M'}$ restricted to
$\pd D\times\pd D$.
Therefore
(cf. \cite[Proposition 1.6]{Iv09} or \cite[Proposition 4.9]{BI})
it admits a 1-Lipschitz extension $\Phi':M'\to\LL$.
For the reader's convenience, here is an explicit formula:
$$
 \Phi'(x)(s) = \inf\{ \Phi(y)(s)+d_{M'}(x,y) :y \in\pd D \} .
$$
Furthermore this extension can be modified so that
$\Phi'(M')\subset\B$ by post-composing it with a cut-off
map $\LL\to\LL$ given by
$$
 \text{cutoff}(\phi)(s) = \min\{R/2,\max\{-R/2,\phi(s)\}\},
 \qquad\phi\in\LL,\ s\in S.
$$
(The cut-off does not affect points in $\Phi'(\pd M')$
since $\Phi'|_{\pd M'}=\Phi|_{\pd D}$ and $\Phi(D)$ is contained
in an $(R/2)$-ball centered at the origin.)
From now on, $\Phi'$ denotes the modified extension.

Consider a map $F=P_\sigma\circ\Phi':M'\to M$.
Since $\Phi'|_{\pd D}=\Phi_{\pd D}$, the first assertion
of Proposition \ref{main prop} implies that
$F|_{\pd M'}=id_{\pd D}$, therefore $F(M')\supset D$.
Then the second assertion of Proposition \ref{main prop}
implies that $\vol(M')\ge\vol(D,g)$.
Thus $(D,g)$ is a minimal filling.

To prove that $(D,g)$ is a strict minimal filling, suppose
that $\vol(M')=\vol(D,g)$. Then $F$ is volume-preserving,
hence by the third assertion of Proposition \ref{main prop}
we have $\Phi'(M')\subset\Phi(M)$.
Therefore $F$ can be written as $F=\Phi^{-1}\circ\Phi'$.
Since $\Phi$ is distance-preserving and $\Phi'$ is
1-Lipschitz, it follows that $F:M'\to M$ is 1-Lipschitz.
Since $F$ is 1-Lipschitz and volume-preserving,
it follows easily that it is an isometry
(cf.\ e.g.\ \cite[Lemma 9.1]{BI}).
Thus $(D,g)$ is a strict minimal filling.
\end{proof}

\begin{proof}[Proof of Theorem \ref{t:rigid}]
Let $D$ and $g$ be as above,
and let $g'$ be a Riemannian metric on $D$
inducing the same boundary distance function
as $g$. Since $(D,g)$ is a strict minimal filling,
it suffices to show that $\vol(D,g')=\vol(D,g)$.
If $(D,g)$ had a convex boundary,
this would follow from the Santal\'o formula
for the volume of  a simple Riemannian metric.
Since we do not assume convexity,
$(D,g)$ may fail to be simple.
However it is  a region in a simple manifold
(namely in a large ball in $M$) and hence
satisfies the 
SGM (Strong Geodesic Minimizing) condition
introduced by C. Croke \cite{Croke91}. Then Lemma 5.1 from \cite{Croke91}
implies the desired equality $\vol(D',g')=\vol(D,g)$.
\end{proof}

The rest of the paper is organized as follows.
In Section \ref{sec-general} we consider
a general class of constructions that work
for any simple manifold $M$ (not necessarily almost hyperbolic).
The main result of this section is
Proposition \ref{p:mean-curvature-zero}
asserting that an isometric image of $M$
in $L^\infty$ is a minimal surface (in the variational sense)
provided that the embedding of $M$ and the Riemannian
structure on $\LL$ used to define the surface area
satisfy certain natural conditions.

In Section \ref{sec-construction} we construct
a distance-preserving map $\Phi:M\to\LL$
(possessing all the nice properties that we need),
a Riemannian structure $G$ in $\B$ used to
define the surface area, and a (preliminary)
projection $P:\B\to M$).
This map does not increase $n$-volumes
in the case of the hyperbolic metric
(that is, when $g=g_0$). 
In the almost hyperbolic case this map
can slightly expand $n$-volumes,
the expansion coefficients are estimated
in Sections \ref{sec-jacobian}
and \ref{sec-correction-factor}.
Finally, in Section \ref{sec-compression}
we compose $P$ with a family of
shrinking maps $M\to M$ so that
the resulting projection $P_\sigma$
strictly decreases $n$-volumes away from
the surface $\Phi(M)$.
This construction completes
the proof of Proposition \ref{main prop}
and the theorems.

\section{General computations}
\label{sec-general}

In this section we do not use the assumption that
our metric is close to the hyperbolic one;
everything here is valid for any Riemannian manifold
$M=(M,g)$.

\label{s:general}
\subsection{Set-up and notation}

Recall that $S=S^{n-1}$ and $\LL=L^\infty(S)$.
We equip $S$ with the standard (Haar) probability measure;
the integral of a function $f$ with respect to this measure
is written as $\int_S f(s)\,ds$.

Let $M=(M,g)$ be an $n$-dimensional Riemannian manifold.
We assume that every two points in $M$ are connected by
a geodesic realizing the distance.
This is always the case if $M$ is complete,
and in the proof of the theorems $M$ is
the hyperbolic space with a perturbed metric.

We denote by $UTM$ the unit tangent bundle of $M$
and by $UT_xM$ its fiber over $x\in M$.
Every sphere $UT_xM$ is equipped with the standard
Haar probability measure; the integral of a function
$f$ with respect to this measure is denoted by
$\int_{UT_xM} f(v)\,dv$.
Note that the integration of these measures on fibers
with respect to the Riemannian volume of $M$
yields the standard Liouville measure on $UTM$.

\begin{definition}\label{d:special-map}
We say that a map $\Phi:M\to\LL$ is a
\textit{special embedding} if
there is a family $\{\Phi_s\}_{s\in S}$ of real-valued
functions on $M$ (that will be referred to as
\textit{coordinate functions} of $\Phi$) such that

(1) For every $x\in M$, the image $\Phi(x)$ is a function
$s\mapsto\Phi_s(x)$ on~$S$ (more precisely, the element of
$\LL$ represented by this function).

(2) The function $(x,s)\mapsto\Phi_s(x)$ is smooth on $M\times S$.

(3) Every function $\Phi_s:M\to\R$ is distance-like,
that is, $|\grad \Phi_s|\equiv 1$ on $M$.

(4) For every $x\in M$, the map $s\mapsto\grad\Phi_s(x)$
is a diffeomorphism between $S$ and $UT_xM$.
\end{definition}

Note that every special embedding $\Phi$ is a distance-preserving map.
Indeed, $\Phi$ is 1-Lipschitz since its coordinate functions are.
To prove that $\|\Phi(x)-\Phi(y)\|_{L^\infty}\ge d_M(x,y)$ for all $x,y\in M$,
consider a unit-speed minimizing geodesic $\gamma$ connecting $x$ to~$y$.
By the 4th condition of Definition~\ref{d:special-map},
there exists $s\in S$ such that
$\grad\Phi_s(x)$ is the initial velocity vector of $\gamma$.
Since $\Phi_s$ is distance-like, its gradient curves are geodesics,
hence $\gamma$ is a gradient curve of $\Phi_s$.
Therefore 
$$
\Phi_s(y)-\Phi_s(x)=\mbox{length}(\gamma)=d_M(x,y) 
$$
and hence $\|\Phi(x)-\Phi(y)\|_{L^\infty}\ge d_M(x,y)$.
Since $\Phi$ is 1-Lipschitz, 
the latter inequality turns into equality
and therefore $\Phi$ is distance-preserving.

\begin{notation}
\label{notation-alpha}
We denote the inverse of the above diffeomorphism
$s\mapsto\grad\Phi_s(x)$ by~$\alpha$.
That is, $\alpha:UTM\to S$ is a map such that
$\grad\Phi_{\alpha(v)}(x)=v$ for all $x\in M$ and $v\in UT_xM$.
By $\alpha_x$ we denote the restriction of $\alpha$ to $UT_xM$,
this map is a diffeomorphism from $UT_xM$ to~$S$.

For every $x\in M$, let $\mu_x$ be the push-forward of the
normalized Haar measure on $UT_xM$ under the diffeomorphism
$\alpha_x:UT_xM\to S$.
Then $\mu_x$ is a smooth probability measure on $S$.
We denote by $\lambda(x,s)$ the density
of $\mu_x$ at $s\in S$ with respect to the
above fixed measure on~$S$.
\end{notation}

The second requirement of Definition \ref{d:special-map}
implies that the image of $\Phi$
lies in the subspace $C^\infty(S)\subset\LL$,
$\Phi$ is a smooth map
(even w.r.t.\ the $C^\infty$ topology in the target),
and its derivative $d_x\Phi:T_xM\to\LL$ is given by
\be\label{e:dPhi}
 d_x\Phi(v)(s) = d_x\Phi_s(v), \qquad v\in T_xM,\ s\in S .
\ee
Let $v_0\in UT_xM$ and $T=d_x\Phi(v_0)$. We rewrite
the above identity as
$$
 T(s) = d\Phi_s(v_0)= \<\grad \Phi_s(x),v_0\> = \<\alpha_x^{-1}(s),v_0\> .
$$
Substituting $s=\alpha(v)$, $v\in UT_xM$ yields
\be\label{e:TPhi}
 T(\alpha(v)) = \<v,v_0\> 
 \qquad\text{if $v_0,v\in UT_xM$ and $T=d_x\Phi(v_0)$}.
\ee
As usual, the image of the derivative $d_x\Phi:T_xM\to\LL$
is denoted by $T_x\Phi$ and referred to
as the tangent space of $\Phi$ at~$x$,
and the elements of $T_x\Phi$ are referred to
as the tangent vectors of $\Phi$.

The density function $\lambda:M\times S\to\R$ is smooth and positive.
The definitions imply that
\be\label{e:int-alpha}
 \int_S f \,d\mu_x = \int_S f(s)\lambda(x,s)\,ds =\int_{UT_xM} f(\alpha(v))\,dv
\ee
for any measurable function $f:S\to\R$.

\begin{definition}
By a \textit{scalar product} on $\LL$ we mean an $L^2$-compatible
scalar product, that is, a symmetric bilinear form $G$ on $\LL$
satisfying
$$
 c\|u\|_{L^2(S)}^2 \le G(u,u) \le C\|u\|_{L^2(S)}^2
$$
for some positive constants $c$ and $C$. The norm of a scalar
product is defined as the minimum possible value of~$C$.

A \textit{Riemannian metric} in an open set $\U\subset\LL$ is
a smooth family $G=\{G_\phi\}_{\phi\in\U}$
of scalar products on $\LL$.
(The smoothness here is that of a map from $\U\subset L^\infty(S)$
to the space of $L^2$-compatible scalar products on $\LL$
equipped with the above norm.)
\end{definition}

\begin{definition}\label{d:special-metric}
Let $G$ be a Riemannian metric in an open set $\U\subset\LL$
containing $\Phi(M)$.
We say that $G$ is \textit{special} with respect to $\Phi$
if the following holds:

(1) For every $\phi\in\U$, the scalar product $G_\phi$
has the form
$$
 G_\phi(X,Y) = n\int_S X(s)Y(s)\, d\nu_\phi(s), \qquad X,Y\in\LL,
$$
where $\nu_\phi$ is a probability measure on $S$.

(2) Every measure $\nu_\phi$ has positive density separated
away from zero;
these densities depend smoothly on $\phi$.

(3) If $\phi=\Phi(x)$ for an $x\in M$, then $\nu_\phi=\mu_x$.
\end{definition}

\begin{lemma}\label{l:G-isometric}
If $G$ is a special Riemannian metric on $\U\subset\LL$ with respect to $\Phi$,
then $\Phi$ is an isometric immersion of $M$ to $(\U,G)$.
\end{lemma}

\begin{proof}
Let $x\in M$, $\phi=\Phi(x)$, $v_0\in UT_xM$, $T=d\Phi(v_0)$. Then
$$
 G_\phi(T,T) = n\int_S T^2\,d\mu_x = n\int_{UT_xM} T(\alpha(v))^2 \,dv
 = n\int_{UT_xM} \<v,v_0\>^2\,dv = 1,
$$
hence the result. The first three equalities in this computation follow
from Definition \ref{d:special-metric}(3), \eqref{e:int-alpha} and
\eqref{e:TPhi}, and the last one is a standard computation on the sphere.
\end{proof}

\subsection{Mean curvature}

\begin{definition}\label{d:II}
Let $G$ be a Riemannian metric in an open set $\U\subset\LL$,
$\Phi:M\to\U$ a smooth surface, $x\in M$,
$V\in\LL$ a vector orthogonal to $T_x\Phi$
with respect to~$G$.
We define a quadratic form on $T_x\Phi$,
called the \textit{second fundamental form}
with respect to $V$ and denoted by $\II_V$, as follows.

Let $F$ be a smooth finite-dimensional submanifold of $\U$
containing a neighborhood of $\phi=\Phi(x)$ in $\Phi(M)$
and such that $V$ is tangent to $F$ at~$\phi$.
Let $T\in T_x\Phi$. Construct a smooth vector field
$\tilde T$ tangent to $\Phi$ near~$\phi$ such that
$\tilde T(\phi)=T$. Then define
$$
 \II_V^{\Phi,x}(T,T) = G(\nabla_T\tilde T,V)
$$
where $\nabla$ is the Levi-Civita connection
of the induced Riemannian metric on~$F$.

The trace of the quadratic form $T\mapsto \II_V(T,T)$
(with respect to the Euclidean structure on $T_x\Phi$
defined by $G=G_\phi$) is called the \textit{mean curvature}
with respect to $V$ and denoted by $H_V^{\Phi,x}$.
\end{definition}

It is easy to see (cf.\ e.g.\ \eqref{e:IIaux} below)
that the second fundamental form does not depend
on the choice of the auxiliary submanifold $F$.
The fact that it does not depend on $\tilde T$
follows from the standard (finite-dimensional)
Riemannian geometry.

\begin{proposition}\label{p:mean-curvature-zero}
Let $\Phi:M\to\LL$ be a special embedding and
$G$ is a special Riemannian metric with respect to~$\Phi$.
Then for every $x\in M$ and every vector $V\in\LL$ orthogonal
to $T_x\Phi$ (with respect to $G$) one has
$$
 H_V^{\Phi,x} = 0 .
$$
\end{proposition}

\begin{proof}
We use the notation from Definition \ref{d:II}.
%Let $T\in T_x\Phi$ be a unit tangent vector
%(that is, $G(T,T)=1$).
Extend $V$ and $T$ to commuting smooth
vector fields $\tilde V$ and $\tilde T$
tangent to $F$ and defined in
a neighborhood of $\Phi(x)$ in~$F$.
In addition, $\tilde T$ should be tangent to $\Phi(M)$.
Then
$$
 \II_V(T,T) = G(\nabla_T\tilde T,V)
 =G(\tilde T,\tilde V)'_T - \frac12 G(\tilde T,\tilde T)'_V
$$
by Riemannian geometry on $F$. Here the notation like
$(\dots)'_X$ denotes the derivative
along a vector $X\in\LL$; this derivative is well-defined
whenever the argument is defined along any smooth curve in $\LL$
with initial direction~$X$. In our case this requirement is
satisfied since $\tilde T$ and $\tilde V$ are defined along~$F$.
Differentiating the above $G$-products yields
$$
 G(\tilde T,\tilde V)'_T
 = G'_T(T,V)+G(\tilde T'_T,V)+G(T,\tilde V'_T)
$$
and
$$
 \frac12 G(\tilde T,\tilde T)'_V = \frac12 G'_V(T,T)+ G(T,\tilde T'_V) .
$$
Since $\tilde T$ and $\tilde V$ commute, we have $\tilde V'_T=\tilde T'_V$, thus
\be\label{e:IIaux}
 \II_V(T,T) = G'_T(T,V)+G(\tilde T'_T,V)-\frac12 G'_V(T,T) .
\ee
We are going to compute the traces of the three terms of this sum separately.
Let $v\in UT_xM$ be such that $T=d\Phi(v)$.
We may assume that the vector field $\tilde T$ is chosen so that
its trajectory through $\Phi(x)$ is the $\Phi$-image
of a constant-speed geodesic $\gamma_v$ in~$M$ such that $\gamma(0)=x$
and $\dot\gamma_v(0)=v$.
Then the first term in \eqref{e:IIaux} takes the form
$$
\begin{aligned}
 G'_T(T,V) 
 &= \frac d{dt}\bigg|_{t=0} G_{\Phi(\gamma_v(t))}(T,V)
 = n\cdot\frac d{dt}\bigg|_{t=0}\int_S T(s)V(s) \lambda(\gamma_v(t),s)\,ds \\
 &= n\int_S V(s) T(s) L_s(v)\,ds
\end{aligned}
$$
where
\be\label{e:defLs}
 L_s(v)=\frac d{dt}\bigg|_{t=0}\lambda(\gamma_v(t),s) .
\ee
(The second identity above follows from the special form 
of $G$ on $\Phi(M)$, cf.\ Definition \ref{d:special-metric}.)
Substituting $T=d\Phi(v)$ and using \eqref{e:dPhi} yields
$$
 G'_T(T,V)= \int_S V(s)\cdot d\Phi_s(v)\cdot L_s(v)\,ds .
$$
Hence
$$
 \trace_{T_x\Phi} \big[T\mapsto G'_T(T,V)\big]
 = n\int_S V(s)\cdot\trace_{T_xM}
 \big[v\mapsto d\Phi_s(v)\cdot L_s(v)\big]\, ds.
$$
Recall that $|\grad\Phi_s|\equiv1$ by the 3rd requirement of Definition \ref{d:special-map}.
Fix an $s\in S$ and choose an orthonormal basis $(v_1,\dots,v_n)$ of $T_xM$
such that $v_1=\grad\Phi_s(x)$.
Then $d\Phi_s(v_1)=1$ and $d\Phi_s(v_i)=0$ for all $i>1$.
Hence
$$
 \trace_{T_xM}
 \big[v\mapsto d\Phi_s(v)\cdot L_s(v)\big]
 = \sum_{i=1}^n d\Phi_s(v_i)\cdot L_s(v_i) = L_s(v_1)
 = L_s(\grad\Phi_s(x)) .
$$
Thus
\be\label{e:IIaux2}
\trace_{T_x\Phi} \big[T\mapsto G'_T(T,V)\big]
 = n\int_S V(s)\cdot L_s(\grad\Phi_s(x))\,ds .
\ee

\begin{lemma}\label{l:Ls-is-Laplacian}
For every $x\in M$ and $s\in S$, one has
$$
 L_s(\grad\Phi_s(x)) =-\lambda(x,s)\cdot\Delta \Phi_s(x)
$$
where $\Delta$ is the Riemannian Laplace operator on~$M$.
\end{lemma}

\begin{proof}
Consider a map $I:M\times S\to UTM$ given by
$I(x,s)=\grad\Phi_s(X)$. Since $\Phi$ is a special map
(cf.\ Definition \ref{d:special-map}(4)),
$I$ is a diffeomorphism;
its inverse $I^{-1}:UTM\to M\times S$ is given by
$I^{-1}(v)=(x,\alpha(v))$ for $v\in UT_xM$.

Consider a vector field $W$ on $M\times S$
given by
$$
 W(x,s) = (\grad\Phi_s(x),0) .
$$
The flow on $M\times S$ generated by this vector field
is mapped by $I$ to the geodesic flow on~$UTM$.
This follows from the fact that every function $\Phi_s$
is distance-like (cf.\ Definition \ref{d:special-map}(3))
and hence the trajectories of its gradient flow are geodesics.

Substituting $v=\grad\Phi_s(x)$ into \eqref{e:defLs} yields
$$
 L_s(\grad\Phi_s(x)) = \lambda'_W(x,s)
$$
where $\lambda'_W$ denotes the derivative of $\lambda$
along $W$.

Let $\mu$ denote the standard product measure on $M\times S$,
that is, the product of the Riemannian volume form on $M$
and the standard measure (the one denoted by $ds$) on~$S$.
Then the measure $\lambda\mu$ (that is, the measure
with density $\lambda$ with respect to~$\mu$) is mapped
by $I$ to the Liouville measure on $UTM$.
Since the Liouville measure is preserved by the geodesic flow,
the measure $\lambda\mu$ is preserved by the flow generated by $W$.
Hence 
$$
\div_{\lambda\mu} W = 0
$$
where $\div$ denotes the divergence (with respect to a given measure).
On the other hand,
$$
 \div_{\lambda\mu}W = \div_\mu W + \lambda^{-1} \lambda'_W .
$$
Thus
$$
 L_s(\grad\Phi_s(x)) = \lambda'_W(x,s)
 = -\lambda(x,s)\cdot \div_\mu W(x,s) .
$$
Since $W(x,s) = (\grad\Phi_s(x),0)$ and $\mu$
is the product measure, we have
$$
 \div_\mu W(x,s) = \div_M \grad\Phi_s(x) = \Delta\Phi_s(x),
$$
and Lemma \ref{l:Ls-is-Laplacian} follows.
\end{proof}

With this lemma, \eqref{e:IIaux2} takes the form
\be\label{e:II-1st}
\trace_{T_x\Phi} \big[T\mapsto G'_T(T,V)\big]
 = -n\int_S V(s)\cdot\lambda(x,s)\cdot\Delta \Phi_s(x) \,ds.
\ee

Now consider the second term of \eqref{e:IIaux}.
Recall that $T=d_x\Phi(v)$ and $\tilde T$ contains
the velocity field of the geodesic $\Phi\circ\gamma_v$.
Hence
$$
 \tilde T'_T = \frac d{dt}\bigg|_{t=0} d\Phi(\dot\gamma_v(t))
 = \frac {d^2}{dt^2}\bigg|_{t=0} \Phi(\gamma(t)) .
$$
The right-hand side is a function on $S$ whose value at $s\in S$
equals 
$$
\frac {d^2}{dt^2}\big|_{t=0} \Phi_s(\gamma_v(t)) = D^2\Phi_s(v,v)
$$
where $D^2$ denotes the Hessian w.r.t.\ the Riemannian metric of~$M$.
Hence
$$
 G(\tilde T'_T,V) = n\int_S D^2\Phi_s(v,v)\cdot V(s)\cdot\lambda(x,s)\,ds
$$
(here we use the special form of $G$ at $\Phi(x)$,
cf.\ Definition \ref{d:special-metric}).
Hence
\be\label{e:II-2nd}
\begin{aligned}
 \trace_{T_x\Phi}\big[T\mapsto G(\tilde T'_T,V)\big]
 &= n\int_S V(s)\cdot\lambda(x,s)\cdot\trace_{T_xM}[D^2\Phi_s(x)]\,ds \\
 &= n\int_S V(s)\cdot\lambda(x,s)\cdot\Delta\Phi_s(x)\,ds .
\end{aligned}
\ee
Adding together \eqref{e:II-1st} and \eqref{e:II-2nd} yields
\be\label{e:II-1+2}
\trace_{T_x\Phi} \big[T\mapsto G'_T(T,V)+G(\tilde T'_T,V)\big] = 0 .
\ee

It remains to get rid of the third term in \eqref{e:IIaux}.
Let $(T_1,\dots,T_n)$ be an orthonormal basis of $T_x\Phi$.
Then
\be\label{e:IIaux4}
 \trace\big[T\mapsto G'_V(T,T)\big]
 = \sum_{i=1}^n G'_V(T_i,T_i)
 = n\sum_{i=1}^n\int_S T_i(s)^2 (\rho'_V)(s) \,ds
\ee
where $\rho=\{\rho_\phi\}_{\phi\in\U}$ is the density
of the measure $\nu_\phi$ from
Definition \ref{d:special-metric} regarded as a function
on $\U$ (with values in $L^\infty(S)$) and $\rho'_V$
is its derivative along~$V$.

Let  $T_i=d_x\Phi(v_i)$ where $v_i\in T_xM$ for $i=1,\dots,n$.
Since $\Phi$ is isometric by Lemma \ref{l:G-isometric},
$(v_1,\dots,v_n)$ is an orthonormal basis of $T_xM$.
Then by \eqref{e:TPhi} we have
$$
 \sum T_i(\alpha(v))^2 = \sum\<v,v_i\>^2 = |v|^2 = 1
$$
for all $v\in UT_xM$. Since $\alpha_x:UT_xM\to S$
is surjective, it follows that $\sum T_i(s)^2=1$ for all $s\in S$.
Then \eqref{e:IIaux4} takes the form
\be\label{e:II-3rd}
\trace\big[T\mapsto G'_V(T,T)\big]
= n\int_S (\rho'_V)(s) \,ds = n\left(\int_S \rho\right)'_V = 0
\ee
since the integral of $\rho$ is fixed to be 1
(cf.\ Definition \ref{d:special-metric}(2)).
Now \eqref{e:IIaux}, \eqref{e:II-1+2} and \eqref{e:II-3rd}
imply that
$$
 H_V^{\Phi,x} = \trace_{T_x\Phi}(\II_V^{\Phi,x}) = 0 -\frac12\cdot 0 = 0.
$$
This finishes the proof of Proposition \ref{p:mean-curvature-zero}.
\end{proof}

\subsection{Jacobians and projections}

\begin{definition}\label{d:linear-jacobian}
Let $G$ be a scalar product on $\LL$, $X$ an $n$-dimensional
Euclidean space and $L:\LL\to X$ a linear map
bounded with respect to $G$
(or, equivalently, with respect to the standard
$L^2$ scalar product).
For an $n$-dimensional subspace $Y\subset\LL$,
we denote by $J_{G,Y}$ the Jacobian
(that is, the $n$-dimensional volume expansion coefficient)
of the restriction $L|_Y$ where $Y$ is regarded as a
Euclidean space whose scalar product is the restriction of $G$.

The ($n$-dimensional) Jacobian of $L$,
denoted by $J_GL$, is the supremum 
of $J_{G,Y}$ where $Y$ ranges over all $n$-dimensional
linear subspaces of $\LL$.
\end{definition}

Obviously $JL=0$ if the rank of $L$ is less than $n$.
If the rank equals $n$, then the supremum is attained
when $Y$ is the orthogonal complement to $\ker L$.
It follows that $J_GL$ depends continuously on $L$ and $G$,
moreover this dependence is smooth on the set where $J_GL\ne 0$.

\begin{definition}\label{d:jacobian}
Let $G=\{G_\phi\}_{\phi\in\U}$ be a Riemannian metric
in an open set $\U\subset\LL$, and let $P:\U\to M$ be an $L^2$-smooth map
(see  Definition \ref{d:L^2-smooth} below).
For an $\phi\in\U$, we define the $n$-dimensional Jacobian
of $P$ at $\phi$, denoted by $J_GP(\phi)$, as the Jacobian
of the derivative $d_\phi P:\LL\to T_{P(\phi)}M$ with respect
to the scalar product $G_\phi$ on $\LL$ and the Riemannian
scalar product on $T_{P(\phi)}M$.

For an $n$-dimensional subspace $Y\subset T_\phi\LL=L$,
we denote by $J_{G,Y}P$ the Jacobian of the restriction
of $d_\phi P$ to~$Y$. 

We will omit $G$ in these notations if the metric is clear from context.
\end{definition}

\begin{definition}\label{d:L^2-smooth}
We say that a map from an open set $\U\subset\LL$ to $M$ is $L^2$-smooth if
it is differentiable (with respect to the $L^\infty$ structure on $\U$), its derivative at every point $\phi\in\U$
can be extended as a bounded linear map from $L^2$ to a fiber of $TM$
and this map from $L^2$ in its turn smoothly depends on~$\phi$. 
\end{definition}

Note that the Jacobian of an $L^2$-smooth map is a continuous function and
moreover it is smooth in the complement of its zeros.

\begin{definition}\label{d:projection}
Let $\Phi:M\to\U\subset\LL$ be a smooth isometric immersion
with respect to a Riemannian metric $G$ on $\U$.
We say that a map $P:\U\to M$ is
a \textit{projection} if it is $L^2$-smooth and

(1) $P\circ\Phi=id_M$ .

(2) For every $x\in M$, $d_{\Phi(x)}P(V)=0$
for every vector $V\in\LL$ orthogonal (with respect to $G$)
to $T_x\Phi$.
\end{definition}

\begin{proposition}\label{p:dJ}
Let $\Phi:M\to\U\subset\LL$ be a smooth isometric immersion
with respect to a Riemannian metric $G$ on $\U$
and $P:\U\to M$ a projection in the sense of Definition \ref{d:projection}.
Then, for every $x\in M$ and every $V\in\LL$ orthogonal to $T_x\Phi$,
$$
 d_{\Phi(x)}J_GP(V) =  H_V^{\Phi,x} .
$$
In particular, if $\Phi$ is a special embedding and $G$ is
a special metric with respect to~$\Phi$, then
$$
d_{\Phi(x)}J_GP(V)=0 .
$$
\end{proposition}

\begin{proof}
It suffices to verify this fact for a finite-dimensional
Riemannian manifold $F$ in place of $(\LL,G)$.
Indeed, it suffices to compute the derivative
of $JP$ along a smooth curve $\gamma$
in $\U$ starting at $\Phi(x)$ with initial velocity $V$.
At every point $\gamma(t)$, there is a unique
$n$-dimensional subspace $Y(t)$ of $\LL$ where
the maximum in the definition of the Jacobian is attained,
and $Y(t)$ depends smoothly on~$t$.
Furthermore, $Y(0)$ is the tangent space of $T_x\Phi$.
Therefore the derivative of the Jacobian can be
computed within a smooth submanifold $F$
containing (locally) $\Phi(M)$ and $\gamma$ and
such that every subspace $Y(t)$ is tangent to $F$
at $\gamma(t)$. It is easy to construct such a submanifold
of dimension $n+1$.

In the finite-dimensional case, there exists a smooth variation
$\{\Phi_t\}_{t\in(-\ep,\ep)}$ of $\Phi$ such that

(1) $P\circ\Phi_t=id_M$ for all $t$;

(2) $\frac d{dt}\big|_{t=0}\Phi_t(x)=V$;

(3) the image of $d_x\Phi_t$ is the subspace
realizing the Jacobian of $P$.

\noindent
The assumptions on $P$ imply that the variation field
$\frac d{dt}\big|_{t=0}\Phi_t$ is orthogonal to our
surface $\Phi$.
Then $J_{\Phi_t(x)}=  J\Phi_t(x)^{-1}$, hence
$$
 d_{\Phi(x)} JP(V) = - \frac d{dt} J\Phi_t(x) .
$$
The right-hand side is the first variation of the
$n$-dimensional surface area which is known
to be equal to $-H_V^{\Phi,x}$.
\end{proof}

\subsection{Surface area}

Let $N$ be a smooth $n$-dimensional manifold
and $f:N\to\LL$ a Lipschitz map.
We say that $f$ is {\em weakly differentiable}
at $p\in N$ if there exists a linear map
$d_p^wf:T_pN\to\LL$ (called the weak derivative of $f$ at $x$)
such that, for every $u\in L^1(S)$ the map
$$
  x \mapsto \langle u, f(x)\rangle : N\to\R 
$$
is differentiable at $p$, and its derivative at $p$ 
is the linear map
$$
 v \mapsto \langle u,d_p^wf\rangle : T_pN\to \R.
$$
Here the angle brackets denote the standard
pairing between $L^1(S)$ and $\LL=L^\infty(S)$.

It can be shown (cf.\ \cite[\S5]{BI}
or \cite[\S3]{Iv09}) that every Lipschitz map
$f:N\to\LL$ is weakly differentiable a.e.\ on $N$.
Moreover if $N$ is a Riemannian manifold
and $f$ is 1-Lipschitz, then $d_p^wf:T_pN\to\LL$ is
a 1-Lipschitz linear map for a.e.\ $p\in N$.

Let $G$ be a Riemannian structure in an open subset
of $\LL$ containing $f(N)$. For every $x\in N$
where $f$ is differentiable, the weak derivative
$d_x^wf$ determines a pull-back nonnegative definite quadratic form
$f^*G(x):=(d_x^wf)^*(G_{f(x)})$.
This quadratic form determines a nonnegative $n$-volume
density on $T_xN$, denoted by $d\vol_G f(x)$.
The $n$-volume of the surface $f$ with respect to $G$ is defined by
$$
 \vol_G(f) = \int_N d\vol_G f .
$$
(There is a minor technical detail to show that the density
under the integral is measurable. We leave this as an
exercise to the reader. A reader who prefers to skip this
exercise can replace the integral by the upper integral
in the above definition).

\begin{lemma}
Let $N$ be an $n$-dimensional Riemannian manifold
and $f:N\to\LL$ a 1-Lipschitz map.
Suppose that $G$ is a special Riemannian metric
(with respect to some $\Phi$, cf. Definition \ref{d:special-metric}).
Then $f$ does not increase $n$-volumes, that is,
$$
 \vol_G (f) \le \vol(N) .
$$
\end{lemma}

\begin{proof}
This follows from the normalization condition
of Definition \ref{d:special-metric}(1), namely
that the scalar product $G_\phi$ at $\phi\in\LL$
has the form
$$
 G_\phi(X,Y) = n\int_S X(s)Y(s)\, d\nu_\phi(s), \qquad X,Y\in\LL,
$$
where $\nu_\phi$ is a probability measure on $S$.

Let $x\in N$ and $\phi=f(x)$. Fix an orthonormal basis
$e_1,\dots,e_n$ in $T_xN$ and define
$X_i=d_x^wf(e_i)$. Since $d_x^wf$ is 1-Lipschitz,
we have
$$
 \|d_x^w f(v)\|_{L^\infty} \le |v|
$$
for any $v\in T_xN$. Substituting $v=\sum v_ie_i$
yields
\begin{equation}
\label{e:vX}
 \|v_1X_1+\dots+v_nX_n\|_{L^\infty} \le \sqrt{v_1^2+\dots + v_n^2}
\end{equation}
for any $v\in\R^n$. This implies that
\begin{equation}
\label{e:Xsquare}
 X_1(s)^2 + \dots + X_n(s)^2 \le 1
\end{equation}
for a.e.\ $s\in S$.
Indeed,  let $A\subset\R^n$ be a countable dense subset of the unit sphere.
Define a map $X\colon S\to\R^n$ by $X(s)=(X_1(s),\dots,X_n(s))$.
By \eqref{e:vX} we have $\langle v, X(s)\rangle\le 1$
for  every $v\in A$ and a.e. $s\in S$.
Observe that 
$$
|X(s)|_{\R^n}=\sup\{  \langle v, X(s)\rangle : v\in A \}
$$
since $A$ is dense in the unit sphere. Therefore
$|X(s)|_{\R^n}\le 1$ for a.e. $s\in S$
and \eqref{e:Xsquare} follows.
Now
$$
 \operatorname{trace} (f^*G(x)) 
 = \sum G_\phi(X_i,X_i) = \sum n \int_S X_i(s)^2\, d\nu_\phi(s)
 \le n \int_S d\nu_\phi(s) =n .
$$
Since $\operatorname{trace} (f^*G(x))\le n$,
we have $\det(f^*G(x))\le 1$, thus the induced
volume density on $T_xN$ is no greater that
the Riemannian volume form.
\end{proof}

\begin{lemma}\label{l-area-jacobian-inequality}
Let $N$ be an $n$-dimensional Riemannian manifold,
and $f:N\to\LL$ a 1-Lipschitz map.
Suppose that $G$ is a special Riemannian metric
in an open set $\U\subset\LL$ containing $f(N)$
and $P:\U\to M$ an $L^2$-smooth map
(cf. Definition \ref{d:L^2-smooth}).
Then
$$
 \vol (P\circ f) \le \int_N J_GP(f(x)) \,d\vol_N(x) .
$$
Here $\vol(P\circ f)$ denotes the ordinary $n$-volume
of a Lipschitz map (or, equivalently, the area of the image
with multiplicities).
\end{lemma}

\begin{proof}
The definitions of weak derivative and  $L^2$-smoothness
imply that $P$ and $f$ satisfy the chain rule:
$$
 d_x(P\circ f) = d_{f(x)}P\circ d_x^wf
$$
for every $x\in N$ where the weak derivative $d_x^wf$ exists.
It follows that the Jacobian of $d(P\circ f)$ is bounded above
by the product of the Jacobians of $d_{f(x)}P$ and $d_x^wf$
(with respect to $G_{f(x)}$),
and the latter Jacobian does not exceed 1 by the previous lemma.
\end{proof}

\section{The construction}
\label{sec-construction}

Now we return to the case of an almost hyperbolic metric.
Recall that  $g_0$ denotes the standard metric on $\HH^n$,
$o\in\HH^n$ is a fixed origin.
By $B_o(r)$ we denote
the ball of radius $r$ in $\HH^n$ centered at~$o$.
Let $S$ be the ideal boundary of $\HH^n$;
we identify $S$ with $S^{n-1}$ in a natural way
(via the unit tangent space at the origin).

Recall that our metric $g$ is extended from
a region $D\subset B_o(R/5)$ to the entire $\HH^n$
so that  $g\equiv g_0$ outside $B_o(R/2)$ and $g$ is close to $g_0$
on $\HH^n$.
We denote $M=(\HH^n,g)$.

To simplify exposition,
we do not track the dependence on $g$ and its
derivatives in our constructions.
We say that a dependence on $g$ is \textit{smooth}
if for every integer $k>0$ there exists an $r>0$
such that this dependence is $k$ times differentiable
with respect to the $C^r$ norm on the space of metrics
(more precisely, on a neighborhood of $g_0$ in
the space of metrics). 

Since $g$ and $g_0$ coincide outside a compact set,
their boundaries at infinity are canonically identified. 
For every $s\in S$,
let $\Phi_s:M\to\R$ be the Busemann function
of a geodesic ray starting at $o$ towards
$s\in \pd_\infty M=S$. We assume here that $g$ is negatively curved (which is always
the case if it  is close to a hyperbolic metric; also recall that the Busemann 
function of a ray $\gamma$ is 
defined as $\Phi_{\gamma}(x)=\lim_{t \to \infty} (d(x, \gamma(t)-t)$.)
Define a map $\Phi:M\to\LL$
as in Definition \ref{d:special-map}, that is,
$$
 \Phi(x)(s) = \Phi_s(x), \qquad x\in M,\ s\in S.
$$

\begin{lemma}
1. $\Phi_x(s)$ depends smoothly on $x$, $s$ and $g$.

2. If $g$ is sufficiently close to $g_0$,
then $\Phi$ is a special embedding in the sense
of  Definition \ref{d:special-map}.
\end{lemma}

\begin{proof}
1. The lemma becomes obvious once one realizes that in our situation 
the Busemann functions, which are usually defined using asymptotic 
constructions, are objects that can be defined in terms of a compact 
region where our metrics may be non-standard. Namely, rather than
using the boundary at infinity, one could use the boundary of the ball $B_o(R)$,
then the Busemann functions
turn into (normalized by their values at $o$) distance function to the
horospheres tangent to this sphere, and these horospheres
are the same for $g$ and $g_0$. After this observation the proof is
straightforward. 

2. The first three requirements of Definition \ref{d:special-map}
trivially follow from the construction and the first part of the lemma.
The last requirement asserting that the map $s\mapsto\grad\Phi_s(x)$ is a diffeomorphism
between $S$ and $UT_xM$ is trivial in the case $g=g_0$,
and then the general case follows from the fact the derivatives
of this map depend continuously on~$g$. 
\end{proof}

We use notations introduced in Section \ref{s:general}
for special maps. 

\begin{lemma}\label{l:lambda-model-case}
If $g=g_0$, then
$$
 \lambda(x,s)=\frac{d\mu_x(s)}{d\mu_o(s)} = e^{-(n-1)\Phi_s(x)}
$$
for all $x\in M$, $s\in S$.
\end{lemma}

\begin{proof}
This straightforward statement can be found in \cite{BCG}.
\end{proof}

Denote by $\B=\B(R)$ the ball of radius $R$ in $\LL$
with respect to the $L^\infty$ norm.

\begin{definition}\label{d:P}
We define a ``projection'' $P:\U\to M$,
where $\U\subset\LL$ is a neighborhood of $\B\cup\Phi(M)$,
as follows.
For every $\phi\in\LL$, introduce a vector field $W_\phi$ on $M$
by
$$
 W_\phi(x) = \int_S e^{n(\Phi_s(x)-\phi(s))} \grad\Phi_s(x)\, d\mu_x(s)
$$ 
and 
let $P(\phi)$ be a point $x\in M$
such that
$$
  W_\phi(x)=0.
$$
The existence of such a smooth map $P$ is shown in Lemma \ref{l:P-basic} below.

\end{definition}

In the case $g=g_0$, the equation says that $x$ is a critical point of the function
$$
F_\phi(x)=\int_S e^{-n\phi(s)}e^{\Phi_s(x)}ds,
$$
as follows from computations below. It is easy to see that $F_\phi$ is convex 
(and actually strictly quadratically convex)
and grows to infinity as $x \to \infty$, and hence it has only one critical point
where the minimum of $F_\phi$ is attained.

\begin{lemma}\label{l:P-basic}
If $g$ is sufficiently close to $g_0$, then there exists a smooth map $P$ 
satisfying Definition  \ref{d:P}  such that

1. $P$ is defined in
a neighborhood of $\Phi(M)$ containing $\B$;

2. $P(\Phi(x))=x$ for all $x\in M$;

3. $P(\B)\subset B_o(R_1)$ for some radius $R_1$
depending only on $n$ and~$R$.
\end{lemma}

\begin{remark}
We do not yet claim that $P$ is $L^2$-smooth
(in the sense of Definition \ref{d:L^2-smooth}).
This will be shown later. 
\end{remark}

\begin{proof}[Proof of Lemma \ref{l:P-basic}]
Substituting $s=\alpha(v)$, $v\in UT_xM$,
where $\alpha(v)$ is defined in Notation~\ref{notation-alpha},
we rewrite the equation $W_\phi(x)=0$ as
$$
 W_\phi(x) = \int_{UT_xM}  e^{n(\Phi_{\alpha(v)}(x)-\phi(\alpha(v)))} v\,dv = 0 .
$$
This equation is trivially satisfied for $\phi=\Phi(x)$,
and so we define $P$ on $\Phi(M)$ by $P(\Phi(x))=x$
to satisfy the second assertion.
This definition is non-ambiguous  since $\Phi$ is
distance-preserving and hence injective.
In the rest of the proof we extend $P$
to a neighborhood of  $\Phi(M)$ containing $\B$,
using a version of the implicit function theorem.

Consider a map $E:\LL\to L^2(S)$ given by
$$
 E(\phi)(s) = e^{-n\phi(s)} .
$$
Note that the set $E(\B)$ is compact in $L^2$.
Let $\phi\in\LL$ and $\psi=E(\phi)$.
Then the equation $W_\phi(x)=0$
takes the form
\be\label{e:Ppsi}
 \int_S \psi(s) e^{n\Phi_s(x)} \grad\Phi_s(x)\,d\mu_x(s) = 0 .
\ee
We are going to prove that there exists a smooth retraction $\tilde P\colon\tilde\U\to E(\Phi(M))$
defined  in a neighborhood $\tilde U$ 
of the set $E(\B)\cup E(\Phi(M))$ in $L^2(S)$
and satisfying the equation \eqref{e:Ppsi} for $x=\tilde P(\psi)$.
The lemma follows from the existence of such $\tilde P$.
Indeed,  $P:=(E\circ\Phi)^{-1}\circ\tilde P\circ E$ satisfies
Definition \ref{d:P} and  assertions 1--3 of the lemma.

Since the original equation $W_\phi(x)=0$
is satisfied for $\phi=\Phi(x)$, $x\in M$, the new equation \eqref{e:Ppsi}
is satisfied for $\psi=E(\Phi(x))$. Therefore we define $\tilde P$
on $E(\Phi(M))$ to be the identity map and this is consistent with \eqref{e:Ppsi}.
We extent $\tilde P$ from $E(\Phi(M))$ in two steps 
using the implicit function theorem.

First we extend it to a neighborhood of $E(\Phi(M))$ in $L^2(S)$.
For this step it suffices to verify that the derivative of \eqref{e:Ppsi}
with respect to $x$ is non-degenerate at every point $(x,\psi)\in M\times L^2(S)$
such that $\psi=E(\Phi(x))$. In order to do this, we rewrite the equation
as follows.
Since \eqref{e:Ppsi} is linear in $\psi$ and is satisfied
for $\psi=E(\Phi(x))$, it is equivalent to
$$
 \int_S (\psi(s)-E\circ\Phi(x)(s)) e^{n\Phi_s(x)} \grad\Phi_s(x)\,d\mu_x(s) = 0
$$
or, passing to the co-tangent bundle,
$$
 \int_S (\psi(s)-E\circ\Phi(x)(s)) e^{2n\Phi_s(x)} d_x(E\circ\Phi)(s)\,d\mu_x(s) = 0 .
$$
This can be interpreted as follows: the vector $\psi-E\circ\Phi(x)\in L^2(S)$ belongs to the
orthogonal complement to the tangent space $T_x(E\circ\Phi)$ of the surface $E\circ\Phi$
%in $L^2(S)$ 
with respect to the scalar product defined by a measure
$e^{2n\Phi_s(x)}d\mu_x(s)$ depending on~$x$. 
Since $E\circ\Phi\colon M\to L^2(S)$ is a smooth embedding and
the scalar products depend smoothly on $x$, 
the assumption of the implicit function theorem is obviously satisfied.
Thus the equation \eqref{e:Ppsi} uniquely defines a smooth retraction
$\tilde P$ in a neighborhood of $E(\Phi(M))$.
Since the equations depends smoothly on $g$, so does $\tilde P$
(again, by the implicit function theorem, now applied in the product
of $M$, $L^2(S)$ and a suitable space of Riemannian metrics on $M$).

The second step is to extend $\tilde P$ to a larger neighborhood of $E(\Phi(M))$
containing  $E(\B)$. Here we use the assumption that $g$ is a small perturbation
of $g_0$. Since $E(\B)$ is compact in $L^2(S)$ and the equation \eqref{e:Ppsi}
depends smoothly on $x$, $\psi$ and $g$, it suffices to verify the assumption
of the implicit function theorem only in the case $g=g_0$.
More precisely, we verify that if $g=g_0$ then for every $\psi\in E(\B)$ the equation \eqref{e:Ppsi}
has a unique solution $x\in M=\HH^n$ and that at this point $x$
the derivative of \eqref{e:Ppsi} with respect to $x$ is non-degenerate.
Then for $g\approx g_0$ the existence of $\tilde P$ and its smooth dependence on $g$
follow by the implicit function theorem in the space 
$M\times L^2(S)\times \{\text{metrics on $M$}\}$.

From now on we consider
the case $g=g_0$. 
Let $\psi\in E(\B)$ and $x\in M$.
Substituting the formula for the density of $\mu_x$
(Lemma \ref{l:lambda-model-case}) into \eqref{e:Ppsi} yields
$$
 \int_S \psi(s) e^{\Phi_s(x)} \grad\Phi_s(x)\,ds = 0,
$$
or, passing to the co-tangent bundle,
\be\label{e:Ppsi2}
 \int_S \psi(s)\, d_x e^{\Phi_s} \,ds= 0 .
\ee
This equation means that $x$ is a critical point of the function
$F_\psi\in C^\infty(\HH^n)$ defined by
$$
 F_\psi = \int_S \psi(s) e^{\Phi_s}ds .
$$
Now it suffices to show that $F_\psi$ has a unique
critical point which is non-degerate.

Observe that 
\begin{equation}
\label{e:D2Phi-conformal}
D^2(e^{\Phi_s}) = e^{\Phi_s}g
\end{equation}
Indeed, since the second
derivative of $e^{\Phi_s}$  along the ``radial'' direction from $s$ is 
that of $e^t$ at $t=\Phi_s(x)$, and in
orthogonal directions it is equal to the curvature of horospheres.

The identity \eqref{e:D2Phi-conformal} means that the functions $e^{\Phi_s}$
and hence $F_\psi$ belong to the subspace
$Z\subset C^\infty(\HH^n)$  of
all functions $u\in C^\infty(\HH^n)$ satisfying the
second order differential equation $D^2_g u = ug$.
Since $\psi\in E(\LL)$,  $\psi$ is an exponential of a bounded function
and hence is positive and separated away from zero.
This and the definition of $F_\psi$ easily imply that
$F_\psi$ is positive and grows to infinity at the ideal boundary of~$\HH^n$.
The positivity of $\psi$ and the equation  $D^2_g F_\psi = F_\psi g$
imply that $F_\psi$ is strictly convex.
Therefore $F_\psi$ has a unique critical point which
is its non-degenerate global minimum.
The lemma follows.
%
%Let $U$ denote the set of all functions $u\in Z$ such that
%$u$ is positive and attains a minimum. The equation  $D^2_g u = ug$
%implies that every $u\in U$ is strictly convex, hence the point where
%$u$ attains its minimum is the unique critical point and depends
%smoothly on $u$. Since $\psi\mapsto F_\psi$ is 
%a smooth map from $L^2(S)$ to $Z$,
%the critical point of $F_\psi$ is unique
%and depends smoothly on $\psi$ on the set of $\psi\in L^2(S)$
%such that $F_\psi\in U$.
%
%If $\psi\in E(\LL)$, then $\psi$ is an exponential of a bounded function
%and hence is positive and separated away from zero.
%From the definition of $F_\psi$ one easily sees that in this case
%$F_\psi$ is positive and grows to infinity near the ideal boundary of $\HH^n$.
%Thus $F_\psi\in U$ for all $\psi\in E(\LL)$ and hence, by continuity,
%for all $\psi$ from a neighborhood of $E(\LL)$.
%The latter contains $E(\B)\cup E(\Phi(M))$.
\end{proof}

\begin{remark}
To help the reader visualize the map $P$ in case $g=g_0$,
we make the following observations.
In the notation from the proof above,
$\dim Z=n+1$ and $\HH^n$ naturally
embeds into $Z$ by the map $x\mapsto F_\psi$ where $\psi=E(\Phi(x))$.
The image $Q$ of this embedding
is a connected component of a hyperboloid of signature $(n,1)$ in $Z$
(and actually is equivalent to the standard hyperboloid model
of $\HH^n$ in $\R^{n,1}$).
Positive functions in $Z$ form a cone over a sheet of this hyperboloid,
and the map $P:\LL\to\HH^n$ can be defined as follows:
First, the argument $\phi$ is transformed into $F_\psi=F_{E(\phi)}\in Z$
(moreover $F_\psi$ lies in the positive cone),
then the resulting point $F_\psi$ is projected radially to $Q$,
and finally this point of $Q$ is mapped back to~$\HH^n$.
\end{remark}

From now on, we reserve the notation $P$ for the map constructed above.
Our next goal is to compute the derivative of $P$.
This is done in Lemma~\ref{l:dP}.

\begin{notation}
For every $x\in M$ and $s\in S$, define a linear operator 
$A_{x,s}:T_xM\to T_xM$ by
\be\label{e:def-Axs}
 A_{x,s}(\xi) = e^{-n\Phi_s(x)} \lambda(x,s)^{-1} \nabla_\xi T_s
\ee
where $T_s$ is a vector field on $M$ given by
\be
  T_s(x) = e^{n\Phi_s(x)} \lambda(x,s) \grad\Phi_s(x) 
\ee
and $\nabla_\xi$ denotes the Levi-Civita derivative along $\xi$.

Let $\phi\in\LL$ and $x=P(\phi)$. Define a function
$\rho_\phi$ on $S$ by
\be\label{e:def-rho}
 \rho_\phi(s) = e^{n(\Phi_s(x)-\phi(s))}
\ee
and let $\ov\rho_\phi$ be the same function normalized
with respect to the measure $\mu_x$:
\be\label{e:def-ovrho}
 \ov\rho_\phi = \frac{\rho}{\int_S \rho\,d\mu_x} .
\ee

Now define a linear operator $A_\phi:T_xM\to T_xM$ by
\be\label{e:def-Aphi}
 A_\phi = \int_S \ov\rho_\phi(s) A_{x,s}\, d\mu_x(s) .
\ee
\end{notation}

\begin{lemma}\label{l:A-model}
If $g=g_0$, then $A_{x,s}=id_{T_xM}$ for all $x\in M$, $s\in S$,
and therefore $A_\phi=id_{T_xM}$ for all $\phi\in\LL$ such that $P(\phi)=x$. 
\end{lemma}

\begin{proof}
Substituting $\lambda(x,s)=e^{-(n-1)\Phi_s(x)}$ (cf.\ Lemma \ref{l:lambda-model-case})
into the definition of $A_{x,s}$ yields
$$
 A_{x,s}(\xi) = e^{-\Phi_s(x)} \nabla_\xi (e^{\Phi_s}\grad\Phi_s)
 =  e^{-\Phi_s(x)} \nabla_\xi \grad (e^{\Phi_s}) .
$$
From \eqref{e:D2Phi-conformal} we have
$ \nabla_\xi \grad (e^{\Phi_s})  = e^{\Phi_s}\xi$,
hence $A_{x,s}(\xi) = \xi$
and the first assertion follows.
Then the second assertion follows from
the fact that $\int_S\ov\rho_\phi\,d\mu_x=1$.
\end{proof}

Note that $A_\phi$ is invertible if $g$ is close to $g_0$ 
since it depends smoothly on $g$ and is
invertible in the case $g=g_0$ (by Lemma \ref{l:A-model}).

\begin{lemma}\label{l:dP}
Let $\phi\in\B$, $x=P(\phi)$ and $\delta\in T_\phi\LL\simeq\LL$.
Then
$$
 d_\phi P(\delta) = A_\phi^{-1}\circ E_\phi(\delta)
$$
where the linear map $E_\phi:\LL\to T_xM$ is given by
\be\label{e:defE}
 E_\phi(\delta) =  n \int_S \delta(s) \ov\rho_\phi(s) \grad\Phi_s(x)\,d\mu_x(s)
\ee
\end{lemma}

\begin{proof}
Let $\xi=d_\phi(\delta)$. Differentiating the definition of $P$ yields
$$
 D_\phi (W_\phi(x))(\delta) + D_x(W_\phi(x))(\xi) = 0
$$
where $D_\phi$ and $D_x$ are the derivatives with respect
to the variables $\phi$ and $x$ where the latter derivative
is computed in normal coordinates
centered at $x$. (In fact, the second term does not depend on the choice
of local coordinates since the vector field vanishes at~$x$.)
Equivalently,
\be\label{e:implicit-aux}
D_\phi(W_\phi(x))(\delta) + \nabla_\xi W_\phi = 0 .
\ee
Substituting the definition of $W_\phi$ yields
$$
\begin{aligned}
 D_\phi(W_\phi(x))(\delta)
 &= -n \int_S \delta(s)  e^{n(\Phi_s(x)-\phi(s))} \grad\Phi_s(x)\,d\mu_x(s) \\
 &= -n \int_S \delta(s) \rho_\phi(s) \grad\Phi_s(x)\,d\mu_x(s) 
\end{aligned}
$$
where the second identity follows from the definition of $\rho$, cf.~\eqref{e:def-rho}.
Observe that
$$
\begin{aligned}
 W_\phi(x)
  &= \int_S  e^{n(\Phi_s(x)-\phi(s))} \grad\Phi_s(x) \lambda(x,s)\,ds \\
  &= \int_S e^{-n\phi(x)} T_s(x) \,ds,
\end{aligned}
$$
hence the second term of \eqref{e:implicit-aux} takes the form
$$
\begin{aligned}
  \nabla_\xi W_\phi &= \int_S e^{-n\phi(s)} \nabla_\xi T_s(x) \,ds
   =  \int_S e^{-n\phi(s)} \lambda(x,s)^{-1}\nabla_\xi T_s(x) \,d\mu_x \\
   &= \int_S e^{n(\Phi_s(x)-\phi(s))} A_{x,s}(\xi)\,d\mu_x
   =  \int_S \rho_\phi(s)  A_{x,s}(\xi)\,d\mu_x.
\end{aligned}
$$
Now  \eqref{e:implicit-aux} takes the form
$$
-n \int_S \delta(s) \rho_\phi(s) \grad\Phi_s(x)\,d\mu_x(s) 
+ \int_S \rho_\phi(s)  A_{x,s}(\xi)\,d\mu_x = 0,
$$
or, equivalently,
$$
n \int_S \delta(s) \ov\rho_\phi(s) \grad\Phi_s(x)\,d\mu_x(s) 
=  \int_S \ov\rho_\phi(s)  A_{x,s}(\xi)\,d\mu_x = A_\phi(\xi) .
$$
Here we divided by the normalizing constant $\int_S\rho\,d\mu_x$
and substituted the definition of $A_\phi$, cf.~\eqref{e:def-Aphi}.
Hence
$$
 \xi = A_\phi^{-1}\left( n \int_S \delta(s) \ov\rho_\phi(s) \grad\Phi_s(x)\,d\mu_x(s)   \right)
$$
and the assertion follows.
\end{proof}

\begin{definition}\label{d:G}
We introduce a Riemannian metric $G$ on $\B$ as follows:
for every $\phi\in\B$, the scalar product $G_\phi$ on $T_\phi\LL=\LL$
is defined by
\be
 G_\phi(X,Y) = n\int_S X(s)Y(s)\ov\rho_\phi(s) \,d\mu_x(s), \qquad X,Y\in\LL,
\ee
where $x=P(\phi)$.
\end{definition}

\begin{lemma}
1. $G$ is a special metric (cf.\ Definition~\ref{d:special-metric})
with respect to $\Phi$.

2. $P$ is a projection (with respect to $\Phi$ and $G$)
in the sense of Definition \ref{d:projection}.
\end{lemma}

\begin{proof}
1. The first two requirements of Definition \ref{d:special-metric}
follow immediately.
To verify the third requirement, recall that $P \circ \Phi=id_M$
and let $\phi=\Phi(x)$ where $x\in M$.
Then $P(\phi)=x$ and from \eqref{e:def-rho} we have
$$
 \rho_\phi(s) = e^{n(\Phi_s(x)-\phi(s))} = e^{n(\Phi_s(x)-\Phi(x)(s))} = e^0 = 1 .
$$
Therefore $\ov\rho_\phi\equiv 1$ and the assertion follows.

2. The fact that $P$ is $L^2$-smooth (cf.\ Definition \ref{d:L^2-smooth})
follows from its ordinary smoothness (cf.\ Lemma \ref{l:P-basic})
and the explicit formula for its derivative (cf.\ Lemma \ref{l:dP}).
The first requirement of Definition \ref{d:projection} follows from Lemma  \ref{l:P-basic}.
To verify the second one, consider $x\in M$, $\phi=\Phi(x)$ and let $\delta\in\LL$
be orthogonal to $T_x\Phi$ with respect to $G_\phi$. By Lemma \ref{l:dP}
it suffices to verify that $E_\phi(\delta)=0$. By \eqref{e:defE} we have
$$
  E_\phi(\delta) =  n \int_S \delta(s) \ov\rho_\phi(s) \grad\Phi_s(x)\,d\mu_x(s)
   = n \int_S \delta(s) \grad\Phi_s(x)\,d\mu_x(s)
$$
since $\ov\rho_\phi\equiv 1$ for $\phi=\Phi(x)$.
A substitution $s=\alpha(v)$, $v\in UT_xM$ (cf.\ \eqref{e:int-alpha}) yields
$$
 E_\phi(\delta) =  n \int_{UT_xM} \delta(\alpha(v)) v\,dv
$$
since $\grad\Phi_{\alpha(v)}=v$, cf.\ the definition of $\alpha$ in Section \ref{s:general}.

On the other hand, the assumption that $\delta$ is orthogonal to $T_x\Phi$ means that
for every $v_0\in T_xM$ 
\begin{multline*}
 0 = G_\phi(\delta,d_x\Phi(v_0)) = \int_S \delta(s)\cdot d\Phi_s(v_0)\ov\rho_\phi(s)\,d\mu_x(s) \\
 = \int_S \delta(s) \<\grad\Phi_s(x),v_0\>\,d\mu_x(s) = \int_S \delta(\alpha(v)) \<v,v_0\>\,dv
 = \<E_\phi(\delta),v_0\> .
\end{multline*}
(Here we again used the fact that $\rho_\phi\equiv 1$ and the substitution $s=\alpha(v)$).
Since $\<E_\phi(\delta),v_0\>=0$ for all $v_0\in T_xM$, we have $E_\phi(\delta)=0$
and the assertion follows.
\end{proof}

\section{Estimating the Jacobian}
\label{sec-jacobian}

Consider the Jacobian $J_YP(\phi)=J_{G,Y}P(\phi)$ of $P$ with respect to~$G$
where $\phi\in\B$ and $Y$ is an $n$-dimensional subspace of $T_\phi\LL=\LL$.
Lemma \ref{l:dP} implies that
$$
 J_{G,Y}P(\phi) = |\det A_\phi|^{-1} J_{G_\phi, Y}(E_\phi) .
$$
Therefore
\be\label{e:Jmain}
 J_GP(\phi) =|\det A_\phi|^{-1} J_{G_\phi}(E_\phi) .
\ee

\begin{proposition}\label{p:JE<1}
For every $\phi\in\B$ the following holds.

1. $J_{G_\phi}(E_\phi) \le 1$.

2. $E_\phi$ is $n$-Lipschitz with respect to $G_\phi$.
\end{proposition}

\begin{proof}
1. We have to prove that $J:=J_{G_\phi,Y}(E_\phi)\le 1$ for every
$n$-dimensional subspace $Y\subset\LL$.
Choose an orthonormal basis $(\delta_1,\dots,\delta_n)$
in $Y$ and an orthonormal basis $(e_1,\dots,e_n)$
in  $T_xM$ such that the matrix $(a_{ij})$ of $E_\phi|_Y$
with respect to these bases is upper triangular
(that is, $a_{ij}=0$ for $i>j$) and $a_{ii}\ge 0$ for all $i$.
Then
$$
 J = |\det(a_{ij})| = \prod_{i=1}^n a_{ii} .
$$
Substituting $s=\alpha(v)$, $v\in T_xM$, yields
\be\label{e:Ephi}
 E_\phi(\delta) =  n \int_S \delta(s) \ov\rho_\phi(s) \grad\Phi_s(x)\,d\mu_x(s)
  = n \int_{UT_xM} \delta(\alpha(v)) \rho(v) v \,dv
\ee
for every $\delta\in\LL$, where
$$
 \rho(v) = \ov\rho_\phi(\alpha(v)) .
$$
Note that
\be\label{e:Jaux1}
 \int_{UT_xM} \rho(v)\,dv = \int_S \ov\rho_\phi(s)\,d\mu_x(s)= 1 .
\ee
Then
$$
 a_{ij} = \<E_\phi (\delta_i), e_j\> 
 = n\int_{UT_xH^n} \delta_i(\alpha(v)) \<v,e_j\> \rho(v) \,dv ,
$$
hence
\begin{multline}\label{e:Jaux}
 J_Y(E_\phi) = \prod_{i=1}^n a_{ii} \le \left(\frac1n \sum_{i=1}^n a_{ii} \right)^n
 = \left( \sum_{i=1}^n \int_{UT_xM} \delta_i(\alpha(v))\<v,e_i\>\rho(v) \,dv\right)^n \\
 \le \left( \frac12 \sum_{i=1}^n  \int_{UT_xM} \delta_i(\alpha(v))^2\rho(v) \,dv
 + \frac12 \sum_{i=1}^n \int_{UT_xM} \<v,e_i\>^2 \rho(v) \,dv \right)^n .
\end{multline}
where the first inequality is the arithmetic-geometric means inequality
and the second one follows from Cauchy--Schwarz.
Since $G_\phi(\delta_i,\delta_i)=1$, we have
\begin{equation}
\label{e:J1}
 \int_{UT_xM} \delta_i(\alpha(v))^2\rho(v) \,dv =
 \int _S \delta_i(s)^2 \ov\rho_\phi(s) \,ds = \frac 1n G_\phi(\delta_i,\delta_i) =\frac1n ,
\end{equation}
hence the first term of the sum in the right-hand side of \eqref{e:Jaux} equals $\frac12$.
By \eqref{e:Jaux1},
\be\label{e:J2}
 \sum_{i=1}^n \int_{UT_xM} \<v,e_i\>^2 \rho(v) \,dv = \int_{UT_xM} |v|^2 \rho(v) \,dv= 1,
\ee
hence the second term of the sum in the right-hand side of \eqref{e:Jaux} equals $\frac12$.
Thus the right-hand side of \eqref{e:Jaux} equals 1, and the assertion follows.

2. The above argument shows that $\sum \<E_\phi(\delta_i),e_i)\le n$
for every orthonormal $n$-frame $\{\delta_i\}$ in $(\LL,G_\phi)$ and every
orthonormal basis $\{e_i\}$ in $T_xM$. It follows that
 $\<E_\phi(\delta),e)\le n$ for any unit vectors $\delta\in(\LL,G_\phi)$ and
 $e\in T_xM$, hence the result.
\end{proof}

\begin{corollary}
If $g=g_0$, then $P$ does not increase $n$-dimensional volumes
(with respect to $G$).
Therefore every compact region in $\HH^n$ is a minimal filling.
\end{corollary}

\begin{proof}
Recall that $J_GP(\phi) =|\det A_\phi|^{-1} J_{G_\phi}(E_\phi)$,
and in the case $g=g_0$ we have $\det A_\phi=1$ by Lemma \ref{l:A-model}.
Hence $J_GP(\phi) \le 1$ by Proposition \ref{p:JE<1},
and the assertions follow.
\end{proof}

\begin{proposition}
\label{p:jacp}
There is a positive constant $c_0=c_0 (R,n)>0$ such that
for every $g$ sufficiently close to $g_0$ the following holds:
if $\phi\in\B$, $Y\subset\LL$ is an $n$-dimensional subspace,
$\delta\in Y$ and $G_\phi(\delta,\delta)=1$, then
$$
J_{G_\phi,Y}(E_\phi) \leq 1 - c_0\|\delta - d\Phi ( E_\phi(\delta) )\|_{L^2(S)}^2 .
$$
\end{proposition}

\begin{proof}
Choose an orthonormal basis $(\delta_1,\dots,\delta_n)$
in $Y$ such that $\delta_1=\delta$.
Then there exists an orthonormal basis $(e_1,\dots,e_n)$
in  $T_xM$ such that the matrix $(a_{ij})$ of $E_\phi|_Y$
with respect to these bases is upper triangular.
We use notations and formulas from the proof
of Proposition  \ref{p:JE<1}.

Let $m=\frac 1n \sum_i a_{ii}$. Then
\begin{multline*}
%\label{tr1}
 m = \sum_{i=1}^n\int_{UT_xH^n} \delta_i(\alpha(v)) \<v,e_j\> \rho(v) \,dv \\
 = \frac12 \sum_{i=1}^n  \int_{UT_xH^n} \delta_i(\alpha(v))^2\rho(v) \,dv
 + \frac12 \sum_{i=1}^n \int_{UT_xH^n} \<v,e_i\>^2\rho(v) \,dv \\
 -\frac12 \sum_{i=1}^n \int_{UT_xH^n} (\delta_i(\alpha(v))-\<v,e_i\>)^2 \rho(v) \,dv
\end{multline*}
By \eqref{e:J1} and \eqref{e:J2}, the sum of the first two terms
in the right-hand side equals 1
(cf.\  the proof of Proposition \ref{p:JE<1}),
hence
\begin{equation}
\label{tr2}
 m = 1 - \frac12 \sum_{i=1}^n \int_{UT_xH^n} (\delta_i(\alpha(v))-\<v,e_i\>)^2 \rho(v) \,dv
\end{equation}
Denote $\xi_i=d_x\Phi(e_i)$. Then, for every $v\in T_xM$,
$$
 \<v,e_i\> = \< \grad\Phi_{\alpha(v)}(x),e_i\> = d_x\Phi_{\alpha(v)}(e_i) = \xi_i(\alpha(v)),
$$
hence
$$
  \int_{UT_xM} (\delta_i(\alpha(v))-\<v,e_i\>)^2 \rho(v) \,dv
  = \int_{UT_xM} (\delta_i(\alpha(v))-\xi_i(\alpha(v)))^2 \rho(v) \,dv
  \ge 2c_1\|\delta_i-\xi_i\|_{L^2}^2
$$
for some $c_1=c_1(n,R)>0$.
Here we used the fact that $\rho$ and the derivatives of $\alpha$ are uniformly bounded.
Then \eqref{tr2} implies that
\begin{equation}
\label{tr3}
 m \le  1 - c_1 \sum_{i=1}^n \|\delta_i-\xi_i\|^2_{L^2} 
  \le 1 - c_1 \|\delta_1-\xi_1\|^2_{L^2} .
\end{equation}

Denote $J=J_{G_\phi,Y}(E_\phi)$, then $J=\prod_{i=1}^n a_{ii}$
as in the proof of Proposition \ref{p:JE<1}.
By the inequality between arithmetic and geometric means,
\begin{multline}
\label{j1}
J^{1/n} =  \left( \prod_{i<j} \sqrt{a_{ii}a_{jj}} \right)^{2/n(n-1)}
\le \frac 2{n(n-1)}\sum_{i<j} \sqrt{a_{ii}a_{jj}} \\
= \frac 2{n(n-1)}\sum_{i<j} \frac{a_{ii}+a_{jj}-(\sqrt{a_{ii}}-\sqrt{a_{jj}})^2}2
= \frac 1n \sum_{i=1}^n a_{ii} - \frac 1{n(n-1)} \sum_{i<j} (\sqrt{a_{ii}}-\sqrt{a_{jj}})^2 \\
= m - \frac 1{n(n-1)} \sum_{i<j} (\sqrt{a_{ii}}-\sqrt{a_{jj}})^2 
\le m - \frac 1{n(n-1)} \sum_{i=1}^n (\sqrt{a_{11}}-\sqrt{a_{ii}})^2 .
\end{multline}
Since $m=\frac 1n \sum a_{ii} \le 1$ by \eqref{tr2}, we have $a_{ii}\le n$ for all $i$.
It follows that
$$
 |\sqrt{a_{11}}-\sqrt{a_{ii}}| \ge \frac 1{2\sqrt n} |a_{11}-a_{ii}| .
$$
Hence
$$
 \sum_{i=1}^n (\sqrt{a_{11}}-\sqrt{a_{ii}})^2 \ge \frac1{4n}  \sum_{i=1}^n (a_{11}-a_{ii})^2 
 \ge \frac14 (a_{11}-m)^2
$$
(the last inequality here follows from the fact that 
$\sum x_i^2 \ge n \big(\frac1n \sum x_i)^2$
where $x_i=a_{11}-a_{ii}$, $i=1,\dots,n$).
This and \eqref{j1} imply
\begin{equation}
\label{j2}
J^{1/n} \le m - \frac 1{4n(n-1)} (a_{11}-m)^2 .
\end{equation}
Since $m\le 1$, we have
$$
 m \le \frac{m+1-(m-1)^2}2
$$
(this inequality is equivalent to $m^2\le m$).
Plugging this into \eqref{j2} yields
$$
\begin{aligned}
 J^{1/n} &\le \frac{m+1}2 - \frac12(m-1)^2 - \frac 1{4n(n-1)} (a_{11}-m)^2 \\
 &\le \frac{m+1}2 - \frac 1{4n(n-1)} ((m-1)^2+(a_{11}-m)^2)
 \le \frac{m+1}2 - \frac 1{8n(n-1)} (a_{11}-1)^2
\end{aligned}
$$
(the last inequality here follows from the obvious one
$x^2+y^2\ge \frac12(x+y)^2$ applied to $x=m-1$ and $y=a_{11}-m$).
By \eqref{tr3} we have
$$
 \frac{m+1}2 \le 1 - \frac{c_1}2 \|\delta_1-\xi_1\|^2_{L^2},
$$
therefore
\begin{equation}
\label{j3}
 J^{1/n} \le 1 - \frac{c_1}2 \|\delta_1-\xi_1\|^2_{L^2} - \frac 1{8n(n-1)} (a_{11}-1)^2
 \le 1 - c_2 (\|\delta_1-\xi_1\|^2_{L^2} + (a_{11}-1)^2) 
\end{equation}
where $c_2 =\min\{\frac{c_1}2,\frac1{8n(n-1)}\}$.
Recall that $E_\phi (\delta) = E_\phi (\delta_1) = a_{11} e_1$
by the choice of our bases, hence
$$
  d_x\Phi( E_\phi(\delta)) = a_{11}\xi_1 .
$$
Therefore
\begin{equation}
\label{j4}
  \|\delta - d_x\Phi \circ E_\phi (\delta)\|_{L^2}^2
   = \|\delta - a_{11}\xi_1\|_{L^2}^2
   \le 2( \|\delta - \xi_1\|_{L^2}^2 + (a_{11}-1)^2 \|\xi_1\|_{L^2}^2 )
\end{equation}
(this is the inequality $\|x+y\|^2 \le 2(\|x\|^2+\|y\|^2)$
applied to vectors $x=\delta - \xi_1$ and $y=(1-a_{11})\xi_1$ in $L^2$).
Recall that $\xi_1=d_x\Phi(e_1)$ is a unit vector in
the space  $(\LL,G_{\Phi(x)})\subset L^2(\mu_x)$,
hence $ \|\xi_1\|_{L^2}^2\le M$ for some $M=M(n,R)>0$.
(Here we use the fact that the densities of the measures $\mu_x$
are uniformly bounded.)
We may assume that $M\ge 2$, then \eqref{j4} implies that
$$
  \|\delta - d_x\Phi \circ E_\phi (\delta)\|_{L^2}^2 \le M(\|\delta_1-\xi_1\|^2_{L^2} + (a_{11}-1)^2)
$$
Plugging this into \eqref{j3} yields
$$
J^{1/n} \le 1 -  c_2M^{-1} \|\delta - d_x\Phi \circ E_\phi (\delta)\|_{L^2}^2 
$$
hence the result.
\end{proof}

\section{Estimating the correction factor}
\label{sec-correction-factor}

In this section we estimate the correction factor
$\det A_\phi$ in \eqref{e:Jmain}.

\begin{proposition}\label{p:A-estimate}
There are positive constants $r$ and $C=C(n,R)$ such that
for every $g$ sufficiently close to $g_0$
and every $\phi\in\B$ one has
$$
 \|A_\phi-I\| \le C\cdot \|g-g_0\|_{C^r} \cdot\|\phi-\Phi(P(\phi))\|_{L^2(S)} 
$$
where $I=id_{T_{P(\phi)}M}$, and
$$
 |\det A_\phi -1 | \le C\cdot \|g-g_0\|_{C^r}^2 \cdot\|\phi-\Phi(P(\phi))\|_{L^2(S)}^2 .
$$
\end{proposition}

\begin{proof}
Fix $\phi\in\B$ and denote $p=P(\phi)$, $I=id_{T_pM}$.
Since $\phi\in\B$, $|\phi(s)|\le R$ for a.e.\ $s\in S$,
and we may assume that $|\phi(s)|\le R$ for all $s\in S$.
Throughout the argument, we denote by $C(n,R)$ various
positive constants depending only on $n$ and~$R$.
By Lemma \ref{l:P-basic}, $p=P(\phi)$ belongs to a ball
of radius $R_1=C(n,R)$ centered at~$o$.
Hence $|\Phi_s(p)|\le C(n,R)$ and 
$C(n,R)^{-1}\le\frac{d\mu_p(s)}{ds}\le C(n,R)$
for all $s\in S$.
We will use these estimates without explicit reference.

For every $t\in[0,1]$, define $\psi_t,\phi_t\in L^\infty(S)$ by
$$
 \psi_t(s) = 1-t+t \cdot e^{n(\Phi_s(p)-\phi(s))}
$$
and
$$
 \phi_t(s) = \Phi_s(p) -\frac1n \log \psi_t(s) .
$$
Then $\psi_t$ is linear in $t$ and
$$
 \psi_t(s) = e^{n(\Phi_s(p)-\phi_t(s))}
$$
for all $t\in[0,1]$.
Note that $\phi_t$ is a smooth function of~$t$,
$\phi_0=\Phi(p)$ and $\phi_1=\phi$.
Therefore $P(\phi_t)$ is defined for all
$t$ from a neighborhood of 0 in $[0,1]$,
as well as for $t=1$.
We are going to study $\det A_{\phi_t}$ as a function of~$t$.

\begin{lemma}
$P(\phi_t)=p$ for all $t$ such that $P(\phi_t)$ is defined.
\end{lemma}

\begin{proof}
By Definition \ref{d:P} of $P$, the identity $P(\phi_t)=p$
is equivalent to
$$
 \int_S \psi_t(s) \grad\Phi_s(p) \,d\mu_p(s)  = 0.
$$
The left-hand side is linear in $t$ and the equality
holds for $t=0$ and $t=1$ (since $P(\phi_0)=P(\Phi(p))=p$
and $P(\phi_1)=P(\phi)=p$).
Therefore it holds for all~$t$.
\end{proof}

\begin{lemma}\label{l:Aphi0}
$A_{\phi_0}=I$.
\end{lemma}

\begin{proof}
Since $P\circ\Phi=id_M$, we have $d_{\phi_0}P\circ d_p\Phi=I$.
By Lemma \ref{l:dP}, we have $d_{\phi_0}P = A_{\phi_0}^{-1}\circ E_{\phi_0}$,
hence it suffices to verify that
\be\label{e:adfiuhvkdfjghkldfsjgnkldsfj}
 E_{\phi_0}\circ d_p\Phi =I .
\ee
Since $\phi_0\in\Phi(M)$, we have $\ov\rho_{\phi_0}\equiv 1$,
and the definition of $E_{\phi_0}$,
cf.\ \eqref{e:defE}, takes the form
$$
E_\phi(\delta) =  n \int_S \delta(s) \grad\Phi_s(p)\,d\mu_p(s)
$$
for all $\delta\in\LL$. Substitute $\delta=d\Phi(v_0)$ where
$v_0\in T_pM$, that is, 
$$
\delta(s)=d\Phi_s(v_0)=\<\grad\Phi_s(p),v_0\> .
$$
This yields
$$
 E_\phi(d\Phi_s(v_0)) = n \int_S \<\grad\Phi_s(p),v_0\> \grad\Phi_s(p)\,d\mu_p(s)
 = n \int_{UT_pM} \<v,v_0\> v \,dv
$$
where the second identity follows by substituting $s=\alpha(v)$,
cf.\ \eqref{e:int-alpha}. By the symmetry under rotations, the
latter integral is a multiple of~$v_0$. To find out the coefficient,
observe that the scalar product of this integral with $v_0$ equals
$$
 n \int_{UT_pM} \<v,v_0\>^2 \,dv = |v_0|^2 .
$$
Thus
$
 E_\phi(d\Phi_s(v_0)) = v_0
$
which proves \eqref{e:adfiuhvkdfjghkldfsjgnkldsfj}.
Lemma \ref{l:Aphi0} follows.
\end{proof}

The definitions imply that $A_{\phi_t}$ is a smooth function of $t$
(in fact, it is a rational function, see below).

\begin{lemma}\label{l:diff-detA}
$\frac d{dt}\big|_{t=0} \det A_{\phi_t}=0$.
\end{lemma}

\begin{proof}
By \eqref{e:Jmain} we have
$$
 \det A_{\phi_t} = \frac{J_{G_{\phi_t}}(E_{\phi_t})}{J_GP(\phi_t)}
$$
At $t=0$, all terms of this formula are equal to~1.
Indeed, $\det A_{\phi_0}=1$ by Lemma \ref{l:Aphi0},
and the Jacobian of the projection $P$ equals 1 at the surface $\Phi(M)$.
Since $P$ is a projection in the sense of Definition \ref{d:projection},
Proposition \ref{p:dJ} implies that the derivative of the denominator at $t=0$ equals~0.
By Proposition \ref{p:JE<1}, the numerator attains its maximum at $t=0$,
hence its derivative at $t=0$ also equals~0. Therefore the derivative
of the fraction is zero.
\end{proof}

The definitions of $A_\phi$ 
(cf.\ \eqref{e:def-rho}--\eqref{e:def-Aphi})
and $\psi_t$ yield that
\be\label{e:Aphi-formula}
 A_{\phi_t} = \frac{A(t)}{b(t)}
\ee
where $A(t)$ is an operator on $T_pM$ given by
$$
 A(t) =  \int_S \psi_t(s) A_{p,s}\, d\mu_p(s)
$$
and 
$$
 b(t) =  \int_S \psi_t(s)\, d\mu_p(s) .
$$
Since $\psi_t$ is linear in $t$, so are $A(t)$ and $b(t)$.
The definition of $\psi_t$ implies that
$$
 C(n,R)^{-1}\le \psi_t(s) \le C(n,R)
$$
for all $s\in S$, hence
\be\label{e:b-estimate}
 C(n,R)^{-1} \le b(t) \le C(n,R)
\ee
for all $t\in[0,1]$. We rewrite $\psi_t$ as
\be\label{e:psi-via-delta}
 \psi_t(s) = 1+ t\delta(s)
\ee
where
$$
 \delta(s) = e^{n(\Phi_s(p)-\phi(s))} - 1 .
$$
Since $|\Phi_s(p)-\phi(s)|\le C(n,R)$,
we have
$$
 |\delta(s)| \le C(n,R)\cdot |\Phi_s(p)-\phi(s)| = C(n,R)\cdot |\Phi(p)(s)-\phi(s)| 
$$
for all $s\in S$. Hence
\be\label{e:delta-estimate}
 \|\delta\|_{L^2} \le C(n,R)\cdot \|\Phi(p)-\phi\|_{L^2} .
\ee
Using \eqref{e:psi-via-delta}, we rewrite $b(t)$ as
\be\label{e:b-formula}
 b(t) =  \int_S (1+t\delta(s))\, d\mu_p(s) = 1 + t \int_S \delta(s)\, d\mu_p(s) .
\ee
In particular, $b(0)=1$ and hence $A(0)=I$ by \eqref{e:Aphi-formula} Lemma \ref{l:Aphi0}.
Similarly, we rewrite $A(t)$ as
$$
 A(t) = \int_S (1+t\delta(s)) A_{p,s}\,d\mu_p(s) 
 = \int_S A_{p,s}\,d\mu_p(s) + t \int_S \delta(s) A_{p,s}\,d\mu_p(s) .
$$
Substituting $t=0$ yields that the first term equals $A(0)=I$, thus
\be\label{e:A-formula}
 A(t) = I +  t \int_S \delta(s) A_{p,s}\,d\mu_p(s) 
 = b(t) I + t\Delta
\ee
where
$$
 \Delta = \int_S \delta(s) (A_{p,s}-I) \,d\mu_p(s) .
$$
(The second identity in \eqref{e:A-formula} uses \eqref{e:b-formula}.)
By Cauchy--Schwarz we have
$$%\be\label{e:Delta-estimate}
 \|\Delta\| \le C(n,R)\cdot \|\delta\|_{L^2}\cdot \|A_{p,s}-I\|_{L^2} 
 \le C(n,R)\cdot \|\delta\|_{L^2}\cdot \|g-g_0\|_{C^r} 
$$%\ee
for some $r$.
The second inequality follows from the fact that
$A_{p,s}$ depends smoothly on $g$ 
and equals $I$ if $g=g_0$ (cf.\ Lemma \ref{l:A-model}).
Substituting \eqref{e:delta-estimate} yields
\be\label{e:Delta-estimate}
 \|\Delta\| \le C(n,R)\cdot \|g-g_0\|_{C^r} \cdot \|\Phi(p)-\phi\|_{L^2}
\ee

By \eqref{e:Aphi-formula} and \eqref{e:A-formula} we have
$$
 A_{\phi_t} = I + \frac t{b(t)} \Delta .
$$
In particular, $A_\phi=A_{\phi_1}=I+\frac 1{b(1)}\Delta$.
Hence
$$
 \|A_\phi-I\| = b(1)^{-1} \|\Delta\| 
 \le C(n,R)\cdot \|g-g_0\|_{C^r} \cdot \|\Phi(p)-\phi\|_{L^2}
$$
by \eqref{e:b-estimate} and \eqref{e:Delta-estimate},
and the first assertion of Proposition \ref{p:A-estimate}
follows.
Furthermore, since $\|\Phi(p)-\phi\|_{L^2}\le C(n,R)$,
we may assume that $\|g-g_0\|_{C^r}$ is so small that
\be\label{e:Delta-small}
 \|A_\phi-I\| \le 1 .
\ee
By Lemma \ref{l:diff-detA},
$$
 0 = \frac d{dt}\bigg|_{t=0} \det A_{\phi_t}
 = \frac d{dt}\bigg|_{t=0} \det\left( I + \frac t{b(t)} \Delta \right) = \trace\Delta
$$
since $b(0)=1$.
Hence
\be\label{e:trace-Delta}
 \trace(A_\phi-I) = b(1)^{-1} \trace\Delta = 0 . 
\ee
We need the following finite-dimensional lemma.

\begin{lemma}\label{l:matrix-estimate}
There is a constant $C=C(n)>0$ such that the following holds.
For every $n\times n$ matrix $A$ such that
$\trace A=0$ and $\|A\|\le 1$,
one has
$$
 |\det(I+A) - 1| \le C\|A\|^2
$$
where $I$ is the identity matrix.
\end{lemma}

\begin{proof}
Indeed, the differential of $\det(I+A)$ on the subspace $\{\trace A=0\}$
is zero at the point $A=0$, hence
$\det(I+A) - 1\le C_1 \|A\|^2$ for some constant $C_1$ and provided that
$\|A\|^2\le r$ for some positive $r>0$ (this follows from the Taylor expansion
for $\det(I+A)$). This proves the inequality in the $r$-neighborhood of $A=0$.
Since $\det(I+A)$ is bounded on the ball $\{\|A\| \leq 1\}$
(since it is continuous and the ball is compact),
the inequality is trivial for $A$ with $r\leq\|A\| \leq 1$,
where one can use $\max \{\det(I+A), \|A\| \leq 1\}/r^2$ for the constant $C$. 
\end{proof}

Now the second assertion of Proposition \ref{p:A-estimate} follows
from \eqref{e:Delta-small}, \eqref{e:trace-Delta} and Lemma \ref{l:matrix-estimate}
applied to $A=A_\phi-I$.
\end{proof}

\begin{corollary}\label{cor:dP-lip}
If $g$ is sufficiently close to $g_0$, then for every $\phi\in\B$
the map $d_\phi P$ is $2n$-Lipschitz with respect to the metric
$G_\phi$ on $T_\phi\LL=L$.
\end{corollary}

\begin{proof}
Recall that $d_\phi P=A_\phi^{-1}\circ E_\phi$, cf.\ Lemma \ref{l:dP}.
The first assertion of Proposition \ref{p:A-estimate}
implies that $\|A_\phi^{-1}\|\le 2$ provided that $\|g-g_0\|_{C^r}$
is sufficiently small.
Since $E_\phi$ is $n$-Lipschitz
by Proposition \ref{p:JE<1},
it follows that $d_\phi P$ is $2n$-Lipschitz.
\end{proof}

\section{A compression trick}
\label{sec-compression}

Define a function $h:\B\to\R_+$ by 
$$
h(\phi)=\|\phi -\Phi(P(\phi))\|_{L^2}
$$
and let $\ep(g)=\|g-g_0\|_{C^r}$ where $r$ is
from Proposition \ref{p:A-estimate}.
Assuming that $\ep(g)$ is sufficiently small,
we rewrite the assertions of Proposition \ref{p:A-estimate}
as follows:
$$
 \|A_\phi^{-1}-I\| \le C(n,R)\ep(g) h(\phi)
$$
and
$$
 |\det A_\phi^{-1}| \le 1+C(n,R) \ep(g) h^2(\phi) .
$$
Since $d_\phi P=A_\phi^{-1}\circ E_\phi$ (Lemma \ref{l:dP})
and $J_{G_\phi}(E_\phi)\le 1$ (Proposition \ref{p:JE<1}),
it follows that
$$
 J_\phi P \le 1+C(n,R) \ep(g) h^2(\phi) \le 2
$$
provided that $\ep(g)$ is sufficiently small.

Observe that the function $h^2$ is smooth on $\B$
and its derivative is given by
$$
 d_\phi h^2(\delta) = 2\<\phi -\Phi(P(\phi)), \delta-d\Phi(dP(\delta))\>_{L^2}
$$
for all $\phi\in\B$, $\delta\in T_\phi\LL=\LL$. Hence
\be\label{e:dh2-estimate}
 |d_\phi h^2(\delta)|\le 2 h(\phi) \|\delta-d\Phi(dP(\delta))\|_{L^2} .
\ee

For a constant $c>0$ define $F_c:L^\infty(S)\to M\times\R_+$
by $F_c(\phi)=(P(\phi),ch(\phi))$.
Note that $F_c$ is smooth on $\B\setminus\Phi(M)$.

\begin{lemma}
\label{product}
For every $R>0$, there exists a $c_1=c_1(n,R)>0$ such that
for every positive $c<c_1$ and every $\phi\in\B\setminus\Phi(M)$,
the $n$-dimensional Jacobian of $F_c$ at $\phi$
with respect to $G$ is bounded above by
$$
 1 + C(n,R) \ep(g) h^2(\phi) .
$$
\end{lemma}

\begin{proof}
Let $Y$ be an $n$-dimensional subspace of $T_\phi\LL$
equipped with our scalar product $G_\phi$.
Denote the $n$-dimensional Jacobian of $dF_c|_Y$ by $J$.
Choose an orthonormal basis $(\delta_1,\dots,\delta_n)$
in $Y$ such that $d_\phi h(\delta_i)=0$ for $i\ge 2$.
Then choose an orthonormal basis $e_1,\dots,e_n$
in $T_{P(\phi)}M$
as in the proof of Proposition \ref{p:jacp},
namely so that the matrix $(a_{ij})$ of
$d_\phi P|_Y$ w.r.t. these bases is upper triangular
and its diagonal elements $a_{ii}$ are nonnegative.
Then
$$
 J = \sqrt{a_{11}^2+t^2}\cdot \prod_{i=2}^n a_{ii}
$$
where 
$$
t = d_\phi(ch)(\delta_1) = \frac c{2h(\phi)} d_\phi h^2(\delta_1) .
$$
By \ref{e:dh2-estimate} we have
\begin{equation}
\label{t1}
 |t| \le c\cdot \|\delta_1 - d\Phi\circ d_\phi P(\delta_1)\|_{L^2}
\end{equation}
By Corollary \ref{cor:dP-lip} we have
\be\label{e:aii<2n}
 a_{ii} \le |d_\phi P(\delta_i)| \le 2n
\ee
for all $i$, hence 
$$
\|\delta_1 - d\Phi\circ d_\phi P(\delta_1)\|_{L^2}
= \|\delta_1 - a_{11}d\Phi(e_1)\|_{L^2}\le C(n,R) .
$$
Therefore we may assume that $c_1$ is so small that \eqref{t1} implies that
$|t|<(2n)^{-n}$.
Consider two cases.

{\it Case 1: $a_{11}<(2n)^{-n}$}.
Then $\sqrt{a_{11}^2+t^2} \le \sqrt2 (2n)^{-n}$, hence
$$
 J= \sqrt{a_{11}^2+t^2}\cdot \prod_{i=2}^n a_{ii} \le  \sqrt2(2n)^{-n} \prod_{i=2}^n a_{ii}
 \le \sqrt2(2n)^{-1} < 1.
$$
Here the second inequality follows from \eqref{e:aii<2n}.

{\it Case 2: $a_{11}\ge (2n)^{-n}$}. Then
$$
J = \sqrt{a_{11}^2+t^2}\cdot \prod_{i=2}^n a_{ii}
\le \left( a_{11}+\frac{t^2}{2a_{11}}\right)   \prod_{i=2}^n a_{ii}
= J_Y P\left(1 + \frac{t^2}{2a_{11}^2}\right) \le J_Y P + (2n)^{2n} t^2 .
$$
Here we used that $J_Y P=\prod_i a_{ii}$ and $J_YP\le 2$.
Since $d_\phi P=A_\phi^{-1}\circ E_\phi$ (Lemma \ref{l:dP}),
we have
$$
J_Y P = \det A_\phi^{-1} J_Y(E_\phi) \le J_Y(E_\phi) + C(n,R) \ep(g) h^2(\phi) .
$$
For the first term we use the estimate
$$
J_Y(E_\phi) \le 1- c_0\|\delta_1 - d\Phi\circ E_\phi(\delta_1)\|_{L^2}^2
$$
from Proposition \ref{p:jacp}, thus
\be\label{t3}
 J \le 1- c_0\|\delta_1 - d\Phi\circ E_\phi(\delta_1)\|_{L^2}^2 + (2n)^{2n} t^2 + C(n,R) \ep(g) h^2(\phi)
\ee
By \eqref{t1} and the triangle inequality in $L^2$,
$$
|t|
% \le c\cdot \|\delta_1 - d\Phi\circ d_\phi P(\delta_1)\|_{L^2}
 \le c ( \|\delta_1 - d\Phi \circ E_\phi(\delta_1) \|_{L^2} 
  + \|d\Phi\circ(d_\phi P-E_\phi)(\delta_1)  \|_{L^2} )
$$
We estimate the second term using Proposition \ref{p:A-estimate}:
$$
 |(d_\phi P-E_\phi)(\delta_1)| =| (A_\phi^{-1}-I)\circ E_\phi(\delta_1) | \le C(n,R)\ep(g) h(\phi)
$$
since $|E_\phi(\delta_1)|\le n$, cf.\ Proposition \ref{p:JE<1}.
Therefore
$$
 |t| \le c\|\delta_1 - d\Phi \circ E_\phi(\delta_1) \|_{L^2} + C(n,R)\ep(g) h(\phi),
$$
hence
$$
 t^2 \le 2c^2 \|\delta_1 - d\Phi \circ E_\phi(\delta_1) \|_{L^2}^2
 +2(C(n,R)\ep(g) h(\phi))^2
$$
Substituting this into \eqref{t3} yields
$$
J \le 1- (c_0-2(2n)^{2n}c^2)\|\delta_1 - d\Phi\circ E_\phi(\delta_1)\|_{L^2}^2
+ C(n,R) \ep(g) h^2(\phi)
$$
We may assume that $c_1$ is chosen so small that $c_0-2(2n)^{2n}c_1^2\ge 0$,
then 
$$
 J \le 1 + C(n,R) \ep(g) h^2(\phi)
$$
and the lemma follows.
\end{proof}

For $t\in[0,1]$ define a ``homothety'' $A_t:M\to M$ by
$$
A_t(x)=\exp_o (t\cdot \exp_o^{-1}(x)) .
$$
Clearly $A_t$ is a smooth map and
it is $t$-Lipschitz due to nonpositive curvature of~$M$.
For a small  $\sigma>0$, define a map
$Q_\sigma:M\times\R_+$ by
$$
 Q_\sigma(x,h) = A_{(1+\sigma h^2)^{-1}} (x) .
$$

\begin{lemma}
\label{shrink}
If $x\in M$ is such that $\dist_M(o,x)<(4\sigma)^{-1/2}$,
then the $n$-dimensional Jacobian of $Q_\sigma$ at $(x,h)$ is not greater
than $(1+\sigma h^2)^{-1}$.
\end{lemma}

\begin{proof}
Due to the Rauch Comparison Theorem, the $n$-dimensional
Jacobian of $Q_\sigma$ does not exceed
that of the similar map for $\R^n$, namely
$$
 (x,h) \mapsto (1+\sigma h^2)^{-1} x, \qquad x\in\R^n, h\in\R_+ .
$$
The latter equals
$$
 (1+\sigma h^2)^{-(n+1)} \sqrt{(1+\sigma h^2)^2 + (2\sigma h|x|)^2} .
$$
If  $|x|$ does not exceed $(4\sigma)^{-1/2}$,
one easily sees that the expression under the square root is not greater
than $(1+\sigma h^2)^3$, hence the result.
\end{proof}

Now define $P_\sigma:\B\to M$ by
$$
 P_\sigma (\phi) = Q_\sigma (P(\phi), \sigma h(\phi))
 = A_{(1+\sigma^3 h^2(\phi))^{-1}} (P(\phi))
$$
where $P$ and $h$ are the same as above.
Note that the second formula implies that $P_\sigma$ is smooth.

\begin{proposition}\label{p-final-jacobian}
For every $R>0$ there exist $\sigma>0$, $c>0$ and $\ep>0$ such that
the $n$-dimensional Jacobian $J(\phi):=J_GP_\sigma(\phi)$ of $P_\sigma$
with respect to $G$ at any point $\phi\in\B$ satisfies
$$
 J(\phi) \le 1 - c\cdot h^2(\phi) .
$$
provided that $\ep(g)=\|g-g_0\|_{C^r}<\ep$.
\end{proposition}

\begin{proof}
Choose $\sigma$ so that $\sigma<c_1(n,R)$ from Lemma \ref{product}
and $(4\sigma)^{-1/2}>\operatorname{diam}(P(\B))$.
 For a point $\phi\in\Phi(M)$, we have
$d_\phi Q_\sigma=d_\phi P$ since $d_\phi h^2=0$,
and therefore $J(\phi)=J_GP(\phi)=1$.

Now consider a point $\phi\notin\Phi(M)$.
The map $P_\sigma$ is a composition of the map $F_\sigma:\B\to M\times\R_+$
whose Jacobian is estimated in Lemma \ref{product}
and the map $Q_\sigma$ whose Jacobian is estimated in Lemma \ref{shrink}.
These estimates yield
$$
 J \le \frac{1 + C(n,R) \ep(g) h^2(\phi)}{1+\sigma (\sigma h(\phi))^2}
 \le 1 - 2c \cdot h^2(\phi)+ C(n,R) \ep(g) h^2(\phi)
$$
for a suitable $c=c(\sigma,n,R)$.
Choosing $\ep<\frac c{C(n,R)}$ yields the desired inequality.
\end{proof}

Now we are in position to complete the proof of Proposition \ref{main prop}.
Let $\ep$ and $\sigma$ be as in Proposition \ref{p-final-jacobian}.
Assuming that $\|g-g_0\|_{C^r}<\ep$, consider the map
$P_\sigma:\B\to M$ constructed above.
By Lemma \ref{l-area-jacobian-inequality},
for any Riemannian $n$-manifold $N$ and any
1-Lipschitz map $f:N\to\B$ we have
$$
 \vol (P_\sigma\circ f) \le \int_N J_GP_\sigma(f(x))\,d\vol_N(x).
$$
Then the inequality for $J$ from  Proposition \ref{p-final-jacobian}
implies that $\vol (P_\sigma\circ f) \le \vol(N)$.
Moreover in the case of equality we have
$h(\phi)=0$ for all $\phi\in f(N)$, hence $f(N)\subset\Phi(M)$.
Thus the map $P_\sigma$ possesses the properties claimed in
Proposition \ref{main prop}.

\bibliographystyle{plain}

\begin{thebibliography}{00}

\bibitem {BI} D. Burago and S. Ivanov,
{\em Boundary rigidity and filling volume minimality of metrics close to a flat one.}
  Ann. of Math. (2) {\bf 171} (2010), no. 2, 1183--1211.

\bibitem {BI95} D. Burago and S. Ivanov,
{\em On asymptotic volume of tori.}
 Geom. Funct. Anal., (5)  {\bf 5} (1995), 800 -- 808.

\bibitem {BI02} D. Burago and S. Ivanov,
{\em On asymptotic volume of Finsler tori,
minimal surfaces in normed spaces, and symplectic filling volume.}
Ann. of Math. (2) {\bf 156} (2002), 891--914.

\bibitem {BI04} D. Burago and S. Ivanov,
{\em Gaussian images of surfaces and ellipticity
of surface area functionals.} Geom. Funct. Anal.,
{\bf 14} (2004), 469-490.

\bibitem {BCG}
G. Besson, G. Courtois and S. Gallot,
{\em Entropies et rigidit\'es des espaces localement sym\'etriques de
courbure strictement n\'egative},
Geom. Funct. Anal., {\bf 5} (1995), 731--799.

\bibitem {Croke91} C.~Croke, {\em Rigidity and the distance
between boundary points}, J. Diff. Geom. {\bf 33} (1991), 445--464.

\bibitem{CK98}                                                                  
C. Croke and B. Kleiner,                                                        
{\em A rigidity theorem for simply connected manifolds without conjugate points},                                                                               
Ergodic Theory Dynam. Systems  {\bf 18} (1998), no. 4, 807--812.

\bibitem{Croke04} C. Croke,
{\em Rigidity theorems in Riemannian geometry},
in ``Geometric Methods in Inverse Problems and PDE Control'',
C. Croke, I. Lasiecka, G.  Uhlmann, and M. Vogelius eds., Springer, 2004.

%\bibitem {CDS} C. Croke, N. Dairbekov and V. Sharafutdinov,
%{\em Local boundary rigidity of a compact Riemannian
%manifold with curvature bounded above},
%Trans. Amer. Math. Soc. {\bf 352} (2000), no. 9, 3937-3956.

\bibitem {Gromov}
M.~Gromov, {\em Filling Riemannian manifolds},
J. Diff. Geom. {\bf 18} (1983), 1--147.

\bibitem {HT} R. D. Holmes and A. C. Thompson,
{\em N-dimensional area and content in Minkowski spaces},
Pacific J. Math. {\bf 85} (1979), 77-110.

\bibitem {Ivanov}
S.~Ivanov. {\em On two-dimensional minimal fillings},
St.-Petersburg Math. J, {\bf 13} (2002), 17--25.

\bibitem{Iv09}
S. Ivanov, {\em Volumes and areas of Lipschitz metrics},
St. Petersburg Math. J. {\bf 20} (2009), no. 3, 381--405.

\bibitem{I10}
S. Ivanov. {\em Volume comparison via boundary distances},
Proc. ICM 2010, vol.2, 769--784.

%\bibitem {Morgan}
%F. Morgan, {\em Examples of unoriented area-minimizing surfaces}, 
%Trans. AMS 283 (1984), 225-237.

\bibitem {Michel}
R. Michel, {\em Sur la rigidit\'e impos\'ee par la longueur des
g\'eod\'esiques}, Invent. Math. {\bf 65} (1981), 71--83.

\bibitem {PU}
L. Pestov and G. Uhlmann, {\em Two-dimensional compact
simple Riemannian manifolds are boundary distance rigid},
Ann. of Math. (2) {\bf 161} (2005), 1093--1110.

\bibitem {Santalo}
L. A. Santal\'o, {\em Integral geometry and geometric probability},
Encyclopedia Math. Appl., Addison-Wesley, Reading, MA, 1976.

%\bibitem {SU}
%P. Stefanov and G. Uhlmann,
%{\em Boundary rigidity and stability for generic simple metrics},
%preprint, {\tt arXiv:math.DG/0408075}, 2004.

%\bibitem {Th}
%A. C. Thompson, {\em Minkowski Geometry},
%Encyclopedia of Math and Its Applications,
%Vol. 63, Cambridge Univ. Press., 1996

\end{thebibliography}

\end{document}